\documentclass[oneside,english,british]{amsart}
\usepackage[T1]{fontenc}
\usepackage[utf8]{inputenc}
\usepackage{color}
\usepackage{enumitem}
\usepackage{amstext}
\usepackage{amsthm}
\usepackage{amssymb}
\usepackage[authoryear]{natbib}

\makeatletter
\numberwithin{equation}{section}
\numberwithin{figure}{section}
\theoremstyle{definition}
\newtheorem{defn}{\protect\definitionname}
\theoremstyle{remark}
\newtheorem{notation}{\protect\notationname}
\theoremstyle{plain}
\newtheorem{ax}{\protect\axiomname}
\theoremstyle{plain}
\newtheorem{assumption}{\protect\assumptionname}
\theoremstyle{plain}
\newtheorem{cor}{\protect\corollaryname}
\theoremstyle{plain}
\newtheorem{prop}{\protect\propositionname}
\theoremstyle{plain}
\newtheorem{lem}{\protect\lemmaname}
\theoremstyle{plain}
\newtheorem{thm}{\protect\theoremname}
\theoremstyle{plain}
\newtheorem*{assumption*}{\protect\assumptionname}
\theoremstyle{remark}
\newtheorem{rem}{\protect\remarkname}

\usepackage{tikz}
\usepackage{pgfplots}
\usepackage{url}
\usepackage{breakurl}
\newtheorem{principle}{Principle}
\newtheorem{belief}{Tacit belief}
\usepackage{enumitem}
\usepackage{underscore}

\makeatother

\usepackage{babel}
\addto\captionsbritish{\renewcommand{\assumptionname}{Assumption}}
\addto\captionsbritish{\renewcommand{\axiomname}{Axiom}}
\addto\captionsbritish{\renewcommand{\corollaryname}{Corollary}}
\addto\captionsbritish{\renewcommand{\definitionname}{Definition}}
\addto\captionsbritish{\renewcommand{\lemmaname}{Lemma}}
\addto\captionsbritish{\renewcommand{\notationname}{Notation}}
\addto\captionsbritish{\renewcommand{\propositionname}{Proposition}}
\addto\captionsbritish{\renewcommand{\remarkname}{Remark}}
\addto\captionsbritish{\renewcommand{\theoremname}{Theorem}}
\addto\captionsenglish{\renewcommand{\assumptionname}{Assumption}}
\addto\captionsenglish{\renewcommand{\axiomname}{Axiom}}
\addto\captionsenglish{\renewcommand{\corollaryname}{Corollary}}
\addto\captionsenglish{\renewcommand{\definitionname}{Definition}}
\addto\captionsenglish{\renewcommand{\lemmaname}{Lemma}}
\addto\captionsenglish{\renewcommand{\notationname}{Notation}}
\addto\captionsenglish{\renewcommand{\propositionname}{Proposition}}
\addto\captionsenglish{\renewcommand{\remarkname}{Remark}}
\addto\captionsenglish{\renewcommand{\theoremname}{Theorem}}
\providecommand{\assumptionname}{Assumption}
\providecommand{\axiomname}{Axiom}
\providecommand{\corollaryname}{Corollary}
\providecommand{\definitionname}{Definition}
\providecommand{\lemmaname}{Lemma}
\providecommand{\notationname}{Notation}
\providecommand{\propositionname}{Proposition}
\providecommand{\remarkname}{Remark}
\providecommand{\theoremname}{Theorem}

\begin{document}
\selectlanguage{british}
\textbf{Preprint}. Published in Foundations of Science

DOI: 10.1007/s10699-018-9573-z
\title{A constructivist view of Newton's mechanics}
\author{H.G. Solari$^{\dagger}$ and M.A. Natiello$^{\ddagger}$}
\address{$^{\dagger}$Departamento de Física, FCEN-UBA and IFIBA-CONICET}
\address{$^{\ddagger}$Centre for Mathematical Sciences, Lund University; }
\begin{abstract}
In the present essay we attempt to reconstruct Newtonian mechanics
under the guidance of logical principles and of a constructive approach
related to the genetic epistemology of J. Piaget and R. García \citep{piag89}.
Instead of addressing Newton's equations as a set of axioms, ultimately
given by the revelation of a prodigious mind, we search for the fundamental
knowledge, beliefs and provisional assumptions that can produce classical
mechanics. We start by developing our main tool: the No Arbitrariness
Principle, that we present in a form that is apt for a mathematical
theory as classical mechanics. Subsequently, we introduce the presence
of the observer, analysing then the relation objective-subjective
and seeking objectivity going across subjectivity. We take special
care of establishing the precedence among all contributions to mechanics,
something that can be better appreciated by considering the consequences
of removing them: (a) the consequence of renouncing logic and the
laws of understanding is not being able to understand the world, (b)
renouncing the early elaborations of primary concepts such as time
and space leads to a dissociation between everyday life and physics,
the latter becoming entirely pragmatic and justified \emph{a-posteriori}
(because it is convenient), (c) changing our temporary beliefs has
no real cost other than effort. Finally, we exemplify the present
approach by reconsidering the constancy of the velocity of light.
It is shown that it is a result of Newtonian mechanics, rather than
being in contradiction with it. We also indicate the hidden assumption
that leads to the (apparent) contradiction. 

\centerline{\today{}}


\textbf{ORCID numbers:} H. G. Solari: \foreignlanguage{english}{0000-0003-4287-1878}
M. A. Natiello: 0000-0002-9481-7454
\end{abstract}

\maketitle

\section{Introduction}

Probably every physics student has wondered what is behind Newton's
laws. Some may even have read the Principia \citet{newt87} looking
for an answer, just to find a few indications in the introduction
followed by a large hiatus taking them directly to Newton's axioms,
with no clue on how to produce them. While some everyday experiences
can be used to grasp the first and third law, the second law appears
to be arbitrary (to say the least). Are there principles operating
behind the laws? Is there a unity among them? What do they hide? These
questions have bugged all scientists exposed to physics as they have
bugged us since we were students. They are relevant for the constructivist
approach known as \emph{second order science} \citep{muel14,liss17}
as well. Are the foundations of physics intuitions of powerful minds
subject to no other rule than an \emph{a posteriori} test of usefulness?
Are they useful pragmatic beliefs subject to no rule of construction
(an idea going back at least to Peirce \citep{peir31,burk46}) to
be replaced only when a new genius handles us a new set of axioms?
Is then reasoning in physics restricted to the deductive form, to
instrumental reason? Or is there room for critical reason searching
for the foundations? Shouldn't a philosophy of physics follow Hegel's
dictum: ``progress in philosophy is rather a retrogression and a
grounding or establishing by means of which we first obtain the result
that what we began with is not something merely arbitrarily assumed
but is in fact the truth, and also the primary truth.'' \citet[par. 101]{hege01}?

On the utilitarian side, reasoning is the most powerful adaptation
of the animal we call human being to the conditions of Nature, for
reason allows humans to foresee the outcomes of events and to act
taking advantage of opportunities, avoiding dangers and unfavourable
conditions. As a consequence, inasmuch \textcolor{black}{as} reasoning
enhances adaptive capabilities, it is a part of natural selection.
This selection took place in front of the difficulties placed by natural
phenomena. 

On the idealistic side, a long-term desire of mankind is to \emph{understand
Nature}. What is implied in it? Understanding Nature is a search for
harmony, it means trying to see how all apparent contingencies (good
or bad) are in fact aspects of different expressions of the totality
implied in Nature (or can be constructed as such). Even in the quest
for understanding elusive matters as mathematical chaos, the intellectual
exercise is not an apology of what exists beyond all possible organisation/understanding
but rather, to find the perspective that allows us to see harmony
in what, at first sight, was perceived as disorder \citet{sola96l}.
This view is related to the concept of \emph{praxis, }which in phenomenology
is considered ``a human activity that is pursued as an end in itself,
that is, a terminal goal of human  living rather than as practical
problem-solving (\emph{techne})'' \citep{heel03} . 

Understanding Nature demands the development of several processes,
being the most immediate one to strip off the particularities that
dress observations (in relation to the matter under discussion), thus
disclosing properties common to them and relations that often have
been named the\emph{ essence} \citep{huss83}. This process is associated
to abstraction. We emphasise here that in these terms, essence is
not absolute but instead relative to the posed questions. 

For example, for any discussion of the biomechanics of a horse, the
colour of its coat is irrelevant, yet if what we are to discuss is
the ``market price'' of the horse, the coat usually plays a significant
role in a predictive theory. The reader may want to distinguish the
act of abstracting away colour when considering biomechanics from
more demanding idealisations requiring some sort of extrapolation
such as e.g., regarding the dynamics of frictionless bodies as the
(unattainable) limit of the behaviour of bodies subject to different
degrees of friction. For this process Galileo used the word \emph{idealisation}
(see e.g., \emph{Discorsi e Dimostrazioni Matematiche Intorno a Due
Nuove Scienze} \citep[day 4]{gali14}).\textcolor{black}{{} }We will
use in the sequel the word \emph{idealisation} without distinguishing
between both types of processes (we notice that Husserl used the word
\emph{ideation} for the process of intuition transforming the observed
into ideas). In this process of idealisation experience is put in
terms of relations between idealised (abstract) objects through synthetic
judgements (Kant). These relations are tentative and are always to
be suspected of being in contradiction with new observations; we call
them \emph{theories}. 

Descartes introduced the expression ``laws of Nature'' referring
to principles of physics concerning the motion of bodies (\emph{Discourse
on method}  \citep[p. 19]{desc15}). The dogma that Nature follows
its own laws, which can then be discovered, soon gave birth to a new
child: Newton's laws, presented in his \emph{Philosophiæ Naturalis
Principia Mathematica}, \citet{newt87}. These developments represent
the beginning of the scientific revolution which includes the mathematisation
of the idealised relations (Galileo) and a reconstruction of the idea
of ``cause'' to be compatible with idealised relations. Yet, Aristotle's
dictum ``we think we do not have knowledge of a thing until we have
grasped its why, that is to say, its cause'' (Physics, \citep{aris50bce}
) remained valid. In this sense, the quest of knowledge maintained
a continuity from the ancient greeks to the modern scientists.

The understanding of causative relations is the long-term programme
of science. However, the relation of cause and effect is not free
of debate. It was tackled by Hume \citep{hume96} as a necessary relation
between contiguous objects. For Hume, causes must always precede effects.
As realised later by Russell \citep{russ12}, cause became a confusing
term seldom used by scientists. He indicated that instead of causes
scientists use ``state functional relations between certain events
at earlier and later times or at the same time\textquotedbl , he
uses the name \emph{determinants} for these events. Both the determinants
and the relations are necessary to answer Aristotle's question: \emph{why
it happens?} and as such, they are Aristotelian causes.

In the crusade for the ``elimination of metaphysics'', Carnap \citep{carn59}
asserts ``we say of a thing or process \emph{y} that it 'arises out
of' \emph{x} when we observe that things or processes of kind \emph{x}
are frequently or invariably followed by things or processes of kind
\emph{y} (causal connection in the sense of a lawful succession).''
Cause in Carnap takes the form of a relation between observed things
or processes, insisting in Hume's old empiricist view. Let us challenge
such extreme empiricist notion. Consider a rectangular cloth suspended
by strings tied to its four corners. If we let a ball rest on the
cloth we observe a deformation of the cloth and the formation of some
wrinkles that evidence tensions. Is the ball the cause of the wrinkles
and the deformation? Or perhaps the wrinkles have attracted the ball?
Such a simple matter is not easy to decide using Carnap's dictum.
The deformation and the presence of the ball resting on the cloth
are simultaneous; there is no precedence in time. However, it has
never been observed that by making wrinkles on a cloth we can attract
balls in the surroundings. Clearly, we say that the ball is the \emph{cause}
of the deformation and the wrinkles (this example is adapted from
\citep[p. 183]{kant87}). Thus, causes can very well be simultaneous
with effects and the decision about what is the cause of a phenomenon
is taken based on reason and on other experiences. But physics requires
more than this type of analysis. Take the ball, lift it and let it
drop. Is there a relation between what makes the ball fall and the
cloth to wrinkle? We all know that there is a relation between not-falling
(staying on the cloth) and wrinkles, and both events point towards
an element left outside in the previous analysis: gravitation. Without
gravitation (or with lesser gravitation) we would observe no (or smaller)
deformation, and no (or slower) fall. Gravitation can be inferred
but not observed (the same can be said of tension). In our observations,
gravitation is acting as a necessary element in our explanation, i.e.,
a cause in Aristotle's perception. Thus, if we want to construct knowledge
we have to abandon the crusade against metaphysics and allow ingredients
that are produced by our mind; in this case elements such as gravitation
that are not sensible but rather inferred (call them meta-sensible
if it pleases the reader). We shall call them \emph{metaphysical}
(where does physics end?). We notice that by so doing the time-order
between causes and effects is erased to a good extent, for gravitation
is there before, during and after the time the ball rests on the cloth,
although what it counts for us is that it was before and during the
experiment. The fall of the ball on the cloth (with the subsequent
coming to rest) and the wrinkles are the result of at least one common
cause. In the empiricist view, metaphysics is not absent but rather
hidden in habits and commonplaces. We cannot account for experience
without some degree of idealisation, i.e., metaphysics. We observe
that putting metaphysics back into science is suggested as well in
\citet{liss17}. 

Our use of cause will indicate relations of the kind described by
Russell unless otherwise stated. For example, we will later state
that ``force is the cause of acceleration'' since force expresses
a functional relation between a change in velocity, time and other
physical characteristics. Notice that in this case there are no \emph{determinants}.

Equipped with reasoning, idealisation and an understanding of cause
in order to address natural phenomena, in the present work we will
try to distinguish the diverse elements that contribute to the construction
of classical (Newtonian) mechanics elaborating from the construction
of reality in childhood as described by \emph{genetic epistemology}.
We find three clear sources for these contributions, namely: (a) requisites
of reason, named \emph{principles} (b) intuitive building blocks that
are constitutive of our understanding after their production in our
early contact with the world (called \emph{axioms}) and (3) provisional
\emph{assumptions} with respect to the organisation of the sensible
world which, as pragmatical beliefs \citet{peir31}, are the main
target of procedures such as falsation (or falsification) \citet{popp59}
and retroduction (often called abduction) \citet{peir31,burk46}.

With the word \emph{principle} we mean a requisite of reason prior
to any theorisation or empiric tests. Two principles will be stated
in the next section, being the \emph{Principle of no Arbitrariness}
(NAP) the most far-reaching. We note on passing that the word ``principle''
has other uses as well (e.g., Newton's third law is also called the
\emph{Principle of Action and Reaction}). As for axioms, they are
broad building blocks that constitute the basis of a theory. In theories
about Nature, axioms are to a large extent motivated by the way human
beings relate to Nature and may be questioned if they fail in this
respect. Axioms are related to intuitions in Husserl's phenomenology
\citep{huss83}, since we proceed first by removing spatio-temporal
determinants and further incorporating the concept (the essence, \emph{Eidos}).
Husserl states ``seeing an essence is precisely intuition''. Within
a mathematical theory, however, the axioms are accepted without need
for proof. Assumptions, in turn, are local decisions in order to limit
the scope of the theory according to the desired goals. Consider for
example: ``in the sequel we consider forces of instantaneous action
at a distance''. The axioms and previous theory may encompass other
types of forces as well, but our local interest in a given context
is delimited by making this (or other) assumption. Mathematical theories
of Nature are usually built up by axioms and assumptions, where axioms
represent our intuition of Nature, our experience, and are the starting
point of empiricist theories. We add here a higher level to this structure
which comes from the subject and it precedes and rules intuitions.
We argue that this principle is a fundamental element in Physics.

This work is the result of an interplay between philosophy, psychology,
mathematics and physics, traditions of thoughts that we do not conceive
as autonomous but rather different realisations of the same general
form that we call reason, this is, reason as it is needed in this
or that field of use (application). The reader not familiar with this
synthetic form, and perhaps not comfortable in all of the particular
forms is invited to skip the unfriendly details in a first reading.
For example, proofs are unavoidable for mathematicians but physicists
will often adopt a ``for the moment I trust you'' attitude in a
first reading. 

The mathematisation of NAP is in Section 2, the application of NAP
to space and time and a new insight into Galileo's transformations
is to be found in Section 3. Sections 4 and 5 show how NAP enters
in Newton's laws. These two Sections contain the most demanding mathematics.
Section 6 is devoted to reflecting about Newton's mechanics with connections
to other criticisms and analyses of Newton's work. Section 7 treats
the speed of light within the Newtonian system. We first show that
what is measured is expected to be independent of the Galilean frame.
Next we uncover a hidden assumption that leads to the well-known problem
between a constant speed of light, Newton's law and the hidden assumption.
Readers mainly interested in this matter may want to skip Sections
4 and 5 (Section 3 is required). Within this discussion the last mathematical
difficulty appears in Subsubsection 7.1.2 where we adopt, for the
sake of the exercise, absolute time and absolute space showing that
even substantivalists \citet{maud93} have options within the Newtonian
world to account for a constant speed of light. This is an example
of how constructivism offers new forms of thinking physical problems.
Finally Section 8 sums up the manuscript.

Natives of psychology and philosophy will have to endure our speaking
of those subjects as foreigners. Further, we do not pretend to be
disciplinary experts in all of the above mentioned fields, rather
we use the disciplined thought to corroborate, support and sharpen
our own (free) thinking. The reader is warned that this work is made
of undisciplined thinking more in the ``prophetic tradition'' than
in the ``scholastic'' one \citet{nisb71}.

\section{No arbitrariness principle}

In any rational construction attempting to logically articulate the
natural world, we are bound to keep elements not based in reason or
evidence outside the construction if possible. Since what is determined
by chance, whim, or impulse, and not by necessity, reason, or principle
is called \emph{arbitrary}\footnote{See for example, The Free Dictionary, https://www.thefreedictionary.com/},
we adopt the attitude of rejecting arbitrary practices. If it were
not possible to achieve a full exclusion of arbitrariness, we must
make sure that the arbitrariness introduced plays no role in answering
valid questions with respect to Nature. While it is certainly possible
to uproot the questions whose answers depend on arbitrary decisions,
in constructing a theory we must strive to avoid the arbitrariness
as much as possible and to keep the valid questions as broad as possible. 

This approach to understanding will be called the \emph{No Arbitrariness
Principle.} Resigning this principle i.e., admitting arbitrariness
as a basic ingredient, would render communication extremely difficult.
The principle delimits the conditions with which knowledge about Nature
is to develop. As such, it is therefore prior to physics and experience
and cannot be questioned empirically. In a theory of natural phenomena,
subsequent axioms –and even more assumptions– are connected to empiric
evidence and may be falsified empirically, thus suggesting the need
for a more adequate theory. 

Struggling to avoid arbitrariness in a physical description immediately
reflects in the invariance of this description with respect to elections
for which we do not have a reasonable foundation. For example, measurements
imply comparisons and the units of measurement are arbitrary to a
good extent, yet velocity will always be written in terms of a ratio
between distance and time. This form is independent of arbitrary choices
of units, although the figure of the velocity changes according to
the chosen units. 

Within the mathematical theory we are about to develop, we will formulate
these ideas as \emph{principles}:

\begin{principle} There is a material world we perceive with our senses (including experiments).
\end{principle}We call this world Nature and we strive to find laws that describe
its perceived organisation. In this context we state the main idea
of this work:

\begin{principle} {\bf [ No Arbitrariness Principle (NAP)]} No knowledge of Nature depends on arbitrary decisions. 
\end{principle}\label{NAP}

In short, NAP establishes that: \emph{If there is no reason for it,
we shall make no difference.} In this simple form it can be recognised\emph{
}as Leibniz' principle of sufficient reason \citet{ball60} put in
a cognitive context\footnote{The principle of sufficient reason (PSR) has made little or no progress
since the Clarke-Leibniz polemic \citet{ball60}. A discussion among
substantivalists and relationalists has re-emerged with Relativity
Theory but in a form in which philosophy is subordinated to the ``success''
of physics. Within such an intellectual disposition it has been  declared
that ``neither the PSR nor the PII enjoys at present unquestionable
philosophical credentials'' \citet{maud93} (PII stands for Principle
of Identity of Indiscernibles). Actually, the discussion appears as
frozen in time, lacking critical contributions and it still revolves
around God. In contrast, NAP stands for a critical as well as operative
view, it has a long tradition in mathematics where the expression
``without lost of generality we can assume...'' is frequent in proofs,
meaning that the results will not depend on our arbitrary choices
but for the sake of the argument a choice must be made. This is: arbitrariness
should not and will not bear any consequence in what is being proved.
The metacognitive instruction received by every student of physics:
``the final result should not depend in your election of units or
your choice of path to reach it'' manifests the same conviction:
arbitrariness has no part in truth.}. 

The arbitrary decisions adopted while constructing a theory can be
viewed as forming a group of transformations operating over the laws
and/or the magnitudes involved in the theory. Each choice of elements
in the set of arbitrary assumptions is used to produce a presentation
of the theory and all presentations must be \emph{equivalent,} i.e.,
conjugated by elements of the group of arbitrary decisions (we will
formalise this idea in mathematical terms in Subsection \ref{subsec:Mathematisation}).

It is important to notice that this approach evolves from the construction
of reality by the child. We quote Piaget on groups: ``There is a
mutual dependence between group and object; the permanence of objects
presupposes elaboration of the group of their displacements and vice-versa.
On the other hand, everything justifies us in centering our description
of the genesis of space around that of the concept of group. Geometrically,
ever since H. Poincaré this concept has appeared as a prime essential
to the interpretation of displacements {[}...{]} But it is necessary
to remember that we shall attribute the widest meaning to this concept
for if, as recent works have shown, the logical definition of the
group is inexhaustible and involves the most essential processes of
thought, it is possible, purely from our psychological point of view,
to consider as a group every system of operations capable of permitting
a return to the point of departure.'' \citet{piag99}

\subsection{Objectivity and intersubjectivity}

If we observe a cat stalking a bird in our garden, we perceive a distance
between cat and bird and we have indications that both cat and bird
perceive something similar, since the cat does not launch the attack
until the distance is short enough and the bird does not fly away
until the cat is threatening close. 

In the next Section we will formalise these perceptions in the concept
of \emph{relative position}, saying that it is \emph{objective }(as
well as cat and bird are), while our perceptions of it are \emph{intersubjective},
meaning that the cat's, bird's and any observer's perceptions can
be put in one-to-one correspondence. In most of the text we will make
no difference between objective and intersubjective. When we make
a difference it corresponds to the fact that for sensible objects
we produce the internal representation by experience or intuition
(\emph{eidetic seeing or ideation} in Husserl) leaving aside contingent
elements, but there are other objects (essence, Eidos) which are the
datum of eidetic intuition \citep{huss83}. Thus we call 'objective'
representations produced by ideation and 'intersubjective' all representations
irrespective of the origin. As such, space is objective but time is
only intersubjective. A corresponding distinction was made by Kant
\citep{kant87} in his discussion of ``the a priori of knowledge''.

The internal representations need not be the \emph{same} but just
corresponding ones. Let $x_{CB}$ be a symbol for the idea of relative
position between cat and bird. We are used to think that the observer
$a$ represents it as $x_{CB}=x_{C}^{a}-x_{B}^{a}$, which is in fact
a synthetic judgement -not a mathematical one- since it proposes a
relation between three different concepts: the positions $x_{C}^{a}$,
$x_{B}^{a}$ of cat and bird relative to the observer (or to a reference
of her/his choice) and the relative position between cat and bird
$x_{CB}$. Since the observer entered the picture by the arbitrary
action of referring positions with respect to her/himself, this expression
must be subject to the no arbitrariness principle, NAP. It is important
to realise that $x_{CB}$ is a symbol that relates to the sensory
perceived world, not an internal mathematical object.

The relation between mathematics and objectivity has been addressed
from several points of view in the past. For example: the objectivity
of mathematics (post-Hilbert), the objectivity of statistical models
\citep{rasc68}, or the problem of meaning and objectivity \citep{hill00}.
None of the points of view which we are aware of attempts to formalise
objectivity in mathematical terms.

\subsection{Mathematisation of objectivity\label{subsec:Mathematisation}}

Consider a set $E$ of concepts/magnitudes and a set $A$ of arbitrary
decisions, necessary for its representation. Let $R_{a}(e)$ be a
representation of the concept $e\in E$ depending on the arbitrary
decision $a\in A$ and let $F$ be a ``natural law'' expressed as
$F(R_{a}(e))=0$. For example, if the concept $e$ is ``distance'',
$R_{a}(e)$ is the real number giving the distance according to $a$
and the ``natural law'' is some mathematical expression involving
this number, such as $R_{a}(e)-C=0$ if the natural law expresses
that this distance has the constant value $C$.
\begin{defn}
\textbf{\label{def:Intersubjectivity}Objectivity}: A law is objective
if $F(R_{a}(e))=0$ holds for any $a\in A$.
\end{defn}
Consider an invertible transformation $T_{ab}$ that maps a representation
$R_{a}$ onto another representation $R_{b}$ ($T_{ab}\circ T_{ba}=Id$).
By NAP, Definition \ref{def:Intersubjectivity} can be restated as:

\begin{equation}
F(R_{a}(e))=0\Leftrightarrow F(T_{ba}R_{a}(e))=F(R_{b}(e))=0.\label{eq:objectivity}
\end{equation}

The transformations $T_{ab}$ form a group\footnote{We notice that in a structural realist work Delhôtel \citet{delh17}
has considered the restricted case of transformations concerning space-time
following the relativistic tradition born out of Poincaré's principle.
He identifies ``a major item on the structural realist agenda: being
able to demonstrate that structural features are precisely and only
those that are preserved across theory change, in one fashion or another,
exactly or approximately'' and later identifying ``equivalence principles''
as one of such structures. What we discuss here is stronger and deeper,
because we account for the reasons why equivalence structures are
required in a more general setting than Poincaré's Relativity Principle
(that will be briefly and critically discussed below). The structure
identified in \citet{delh17} is a particular case of the mathematisation
of NAP presented here. \foreignlanguage{english}{The requirement of
having a group of transformations connecting descriptions differing
in arbitrary choices is recognised in} \citet[(sec. 3)]{delh17} not
as a requisite of reasoning but only as a consequence of the postulates
of the dynamics.} $T$. We assume that: 
\begin{enumerate}
\item $T_{ab}\circ T_{bc}=T_{ac}$, there exists a composition law.
\item $T_{aa}=Id$, there exists an identity. 
\item $T_{ab}=(T_{ba})^{-1}$or $(T_{ba}\circ T_{ab})=Id$, an inverse exists.
\item $T_{ab}\circ(T_{bc}\circ T_{cd})=(T_{ab}\circ T_{bc})\circ T_{cd}$,
the composition law is associative. 
\end{enumerate}
As a consequence, $T_{ab}\circ(T_{bc}\circ T_{ca})=Id=(T_{ab}\circ T_{bc})\circ T_{ca}$
and $T_{ab}\circ(T_{bc}\circ T_{cd})=T_{ab}\circ T_{bd}=T_{ad}=T_{ac}\circ T_{cd}=(T_{ab}\circ T_{bc})\circ T_{cd}$. 

\section{Space, Time and Observers}

The system of concepts that allows the organisation of reality is
dialectical, with multiple concepts coming into existence at once.
Certain concepts in this system are meaningful only in relation to
other concepts introduced simultaneously. In all dialectical constructions,
the tension of the opposites is the engine of understanding, but the
terms in opposition present no difference between them other than
being just the terms of the opposition. Hegel's dialectic of being
and not-being is perhaps the clearest example of this constructive
procedure \citep{hege01}. In our case, the concepts: ego, alter,
identity, object, space and time are the fundamental building blocks
for the construction of understanding; they cannot be referred to
previous concepts because \emph{understanding} can only be referred
to its opposite: \emph{not-understanding}. Thus, any attempt at defining
these terms will never be completely satisfactory since we will have
to resort to some complicity from the intuitions of the reader. We
call this process \emph{dialectical openings to understanding}\footnote{In our daily quest to understand nature we make use of ideas and concepts
previously produced by our society that we have \emph{naturalised}
(included in Nature although they belong to our cultural baggage).
However, the recursion to previous understanding has an end and it
is not always possible. When we arrive to such a dead end while exercising
the critical reason in the search for the fundamentals we have exhausted
the possibilities of analysis, the limits of Analytical science \citet[(par. 1720-1764)]{hege01}.
Hegel states ``Synthetic cognition aims at the comprehension of what
is, that is, at grasping the multiplicity of determinations in their
unity. It is therefore the second premise of the syllogism in which
the diverse as such is related'' . ``Synthetic cognition'' is presented
by Hegel in contradistinction to ``Analytic cognition''. Kant \citet{kant87}
considered time an a-priori of knowledge, and it is certainly an a-priori
for the adult as well as space is. Piaget taught us that the genesis
of the notions of space and time is found in the early childhood \citet{piag99}.
Accounting for this philosophical tradition, we call \emph{dialectical
openings to knowledge} the dialectical constructions that set up the
basic elements for the knowledge of Nature, such as time, space and
object. In such a way, we emphasize their synthetic origin.}. In practical terms, when we mention a concept (in this section)
using a word that is defined later in the text, it should be understood
in its intuitive form, later to be formalised in a compatible form.
We can identify the dialectical pairs ego-alter and change-permanence
(time-identity). As for space, considering that it implies the concept
of distance and with it the idea of what is within reach and what
belongs to our world but cannot be reached without effort, it is the
result of the dialectic \emph{near me} (reachable) and \emph{at a
distance} (not reachable, i.e., not near me). Notice that the notion
of oppositions as first principles is a subject already considered
in Aristotle's \emph{Physics} \citep{aris50bce}, also referring to
previous philosophers.
\begin{notation}
In this Section we will try as much as possible to use superindices
$a,b$ to indicate subjective (i.e., depending on arbitrary decisions)
concepts, quantities, etc., while related objective entities will
have no superindex.
\end{notation}

\subsection{Space and time}

\subsubsection{Space}

Space is spontaneously conceived by children along the construction
of reality \citep{piag99}. It is not possible to speak separately
of object, space, time and ego, because the construction produces
all of them as a single dialectic system (we could try to straighten
this matter up philosophically, although this would not be loyal to
the child's structuring of thought). By exercising our memory we perceive
some degree of permanence of ourselves, \emph{ego}, and thus an idea
of identity. The real world acquires continuity, we do not longer
watch the movie frame by frame. Together with the recognition of ourselves
comes \emph{alter}, i.e., \emph{not-ego}, and with the help of our
memory some sort of permanency of \emph{alter} emerges as well, the
object that later, stripped of all other characteristics will be the
``body'' referenced in physics. Therefore, objects, people and other
things acquire some sort of identity, inasmuch as we remove some attributes
of them, mainly positions in space. Then, if an object (or ourselves)
is in a place in one movie-frame and in a different place in another
movie-frame, we no longer say they are different objects but rather
that there was a \emph{change} in the positions as well as some \emph{permanence}:
the object itself. The sequence of changes in places is the primary
idea of time. Much later the child will conceive herself as being
of the same condition than family, pets and toys, this is, she will
place herself in the space. Space, with its implication of distance
(a concept that is easy to root in sensorymotive intelligence) is
opposed to the unity of the cosmos of the child.

The perceived space-time, centered in the observer (\emph{ego}) always
has \emph{ego} distinguishing a reference point. This primary, subjective,
notion of space is compatible with an empty space, a holder of objects,
an idea that in physics goes back, at least, to Newton. Emptiness
is in fact a resource for the subsequent suppression of the observer.
By NAP, the required objectivity of the laws of physics manifests
in that no law is objective if it depends on \emph{ego}. The mathematisation
of this idea is expressed by eq.(\ref{eq:objectivity}), where the
group $T$ corresponds to spatial translations and rotations. 
\begin{ax}
A set of three orthogonal directions in the real (sensible) space
is selected and represented by the symbols $\mathbf{e}_{i}$. Any
objective position is then represented by $x_{AB}=\sum_{i}x_{i}\mathbf{e}_{i}$
with $x_{i}$ real numbers.
\end{ax}
Objective space is hence built following Descartes and is represented
as a vector space. Any other choice of reference vectors can be related
to the initial one by a linear transformation, $\mathbf{e}_{j}^{\prime}=\sum_{i}R_{ji}\mathbf{e}_{i}$.
It follows that $\sum_{j}x_{j}^{\prime}\mathbf{e}_{j}^{\prime}=\sum_{i}x_{i}\mathbf{e}_{i}$
for any objective position, hence $x_{j}^{\prime}=\sum_{i}\left[R^{-1}\right]_{ij}x_{i}$.
The observer can then select a reference point and three independent
orientations to describe the position of objects in the world, yet
leaving the distance as objective. The class of equivalence thus generated
is described by the group $IO(3)$, the group of isometries (translations,
rotations and reflections).

Subjective space, being empty, provides no way of distinguishing one
direction from another or one reference point from another. As a consequence
of NAP, subjective space is then \emph{isotropic} and \emph{homogeneous}.
We notice that objective space bears some relation to Leibniz' relational
space, while subjective space is related to Newton's relative space
\citet{ball60}.

We recall that Axioms correspond to intuitions based upon experience,
the Cartesian space agrees with the experience of the child and of
the scientist in Newton times. Furthermore, the analytical properties
of e.g., Riemannian spaces are based upon differentiable manifolds
which in turn rest upon collections of local charts of Cartesian type,
a fact that speaks loudly about which one is the intuitive space.

\subsubsection{Time }

Unlike space, time is undoubtedly related to changes, sequences of
changes, rapidity and causal relations \citet{piag99}. This is the
genetic episteme of time; to call ``time'' any other kind of object
is simply to ask for confusion\footnote{\textcolor{black}{Piaget arrived to his view about the genesis of
the concept of time after a series of experiments. The experiments
are conveniently collected in \citep{piag78}.}

The notion of speed, at least in the form of faster and slower, precedes
the notion of time: ``From the point of view of immediate experience,
the child succeeds very soon in estimating speeds of which he has
direct awareness, the spaces traversed in an identical time or the
“before” and “after” in arrival at a goal in cases of trajectories
of the same length. But there is a considerable gap between this and
a dissociation of the notion of speed to extract a measurement of
time, for this would involve replacing the direct intuitions peculiar
to the elementary accommodation of thought to things by a system of
relations involving a constructive assimilation.'' \citet[p. 383]{piag99}.

Aristotle relates changes in position, velocity and time. In modern
notation we could write (for finite $v>0$)
\[
t-t_{A}=\int_{x_{A}}^{x}\frac{v\cdot dx}{|v|^{2}}=\int_{x_{A}}^{x}\frac{|dx|}{|v|}.
\]
We can move back to the perception of the child by making $v$ constant.
Thus, change in position offers a clock, and since space is assumed
to be continuous, so is time (an old reasoning already present in
Aristotle). As long as the moon revolves around the Earth, we have
change and we have time. No interval of time will ever be empty of
changes. 

\textcolor{black}{Other points of view regarding time have been put
forward and continue to exist. For example, time has been conceived
geometrically as well, taking such character in Special Relativity.
Under the geometrical conception McTaggart has even argued that time
is unreal \citep{mcta08}. The incompatibility of the perceived time
with the time of Relativity resulting from Gödel work \citep{gode49}
was discussed and asserted by Bell \citep{bell02}. Here, a distinction
has to be made. On one hand Piaget's view seeks the genesis of time
(and the way we consider time in our daily facts) in the phenomena
and its apprehension as idea. On the other hand, there are instrumental
beliefs whose roots are nurtured by the mud of speculative material
entities such as the }\textcolor{black}{\emph{æther}}\textcolor{black}{{}
and the }\textcolor{black}{\emph{electrical fluid}}\textcolor{black}{{}
veiled to us by mathematical axioms. The different ideas revolving
around the notion of time might originate in overlooked assumptions
of the idealisation method: when we proceed through successive idealisations,
the order among them matters. Dynamical instabilities and irreversibility
(entropy) are the result of first considering the limit of sufficiently
long times and }subsequently\textcolor{black}{{} (if needed) the limit
of infinite measurement precision \citep{arno68}. Deterministic,
time-reversible dynamics rests on taking these limits in reverse order.
The school of Brussels showed that, for some systems, irreversibility
and time-reversal possibilities }respond to different choices of representation\textcolor{black}{{}
\citep{gold81} connected by transformation groups. The truth is that
we have access to finite precision and finite time lapses, while the
rest is only part of the processes of ideation. The arbitrariness
that NAP tries to remove may take subtle and unexpected shapes.}}. Time is the word we use to express our perception of change; it
is change in its most abstract form. Aristotle indicates ``without
change there is no time'' \citet[Book IV, Ch. 11]{aris50bce}.

The perception of change we use to construct our intuition of time
rests on natural processes. An invariant characteristic of our perception
of Nature is that experienced processes have \emph{beginning} and
\emph{end}. Husserl \citet[Matter of fact]{huss83} states ``The
founding cognitional act of experiencing posit something real \emph{individually;}
they posit it as something factually existing spatio-temporally, as
something that is at \emph{this} temporal locus, that has this duration
of its own and a reality content which, with respect to its essence,
could just as well have been at any other temporal locus.'' 

Time is a verbalisation of the dialectical tension of the pair being/not-being.
In principle, time is associated to objects (or \emph{measured}) by
the order of the sequence of changes (between sundown and sunrise
the rooster sings, and if the rooster wakes me up every morning, it
sings before -or rather while- I wake up, because this song is (part
of) the \emph{cause} of my awakening).  \citet[Ch. 4]{piag99} discusses
the development of the notion of time in the child beginning from
``As early as his reflex activity and the formation of his first
habits, the nursling shows himself capable of two operations which
concern the elaboration of the temporal series. In the first place,
he knows how to coordinate his movements in time and to perform certain
acts before others in regular order. For instance, he knows how to
open his mouth and seek contact before sucking, how to steer his hand
to his mouth and even his mouth to his thumb before putting the thumb
between his lips, etc.'' The child proceeds (always according to
Piaget) in five stages to develop his notion of time, always based
on series. Starting at the third stage memory plays an increasingly
relevant role.

One striking difference between time and space from the psychogenetic
point of view is that while for the change of position in space of
an object there is a possible operation that reverts the change, there
is no operation capable to revert the temporal order of events. Unlike
distance, that reaches us as a perception resulting from the telemetry
of our binocular sight, time does not reach us as a perception, but
it is rather the result of memory, a log of changes and a logical
process that discriminates between the relative order of events. It
follows that time is measured by comparing sequences of changes. By
NAP, any transformation relating time-perceptions of different references
(\emph{ego}s) is constrained to preserve this ordering. Therefore,
while arbitrary individual subjective time may differ among observers,
they are all related by strictly monotonic (bijective) mappings. The
underlying group is the set of strictly increasing continuous functions
$f:R\mapsto R$ with standard composition of functions as the group
product.
\begin{ax}
\label{time} The order between the events (the various determinants
and effects) involved in a causative relation is fixed.
\end{ax}
This axiom relates to the concept of ``determinism'' discussed around
\emph{Equation (A)} in \citep[p. 199]{russ12}. In simpler words,
first the vase falls, then it breaks. It is a vase as long as there
exists a particular cohesion in the material. The alteration produced
by the impact reorders the material in smaller pieces. Hence, without
fall there is no impact, without impact no change in material stress.
Causes here are gravitation and the change in material stress. Determinants
are fall and impact, while breakup (rupture) is an effect. 

When we measure the times of a phenomenon, say the change in position
of a body, we use changes not involved in the process as references
for time. Indeed, we resort to the idealised and imagined order of
all events of the universe, encompassing all possible changes, leaving
no ``room'' between them since a time without changes is a contradiction
in the terms\footnote{If time is conceived as holding a series of events, we may ask whether
it is possible to have a time interval without changes or events (since
each event represents a change). During a time without changes there
would be no heart-beats of the observer, no movement of any object
in the universe, no change that could be used as clock, etc. Between
the two encompassing events of time without change there would be
no way to measure time since measuring it requires something that
changes. We would have then proposed that measurable (sensible) time
differs from ``true time'' by an unaccountable amount, yet all observers
will measure zero time identically since it is a property of monotonic
functions to map open intervals into open intervals. Hence, a time
interval without changes contradicts the early determination of time
as associated to movement (change of position in space) and being
therefore continuous. }. In mathematical terms, time is well represented by a one dimensional
mathematical space such as the real numbers, a notion that we have
been using and will continue to use, along this work. Time-intervals
are referred to this ``background'' of events that are present irrespective
of the phenomenon in study. This background of changes constitutes
a \emph{clock}. We try to use as clocks those devices or observations
that appear to us as regular. As long as we cannot present evidence
that a process runs faster in one circumstance or another, we expect
the relative order between changes in the clock and changes in the
phenomena to remain the same, hence
\begin{assumption}
An objective time can be defined by convening on a process to define
a time-unit.
\end{assumption}
Absolute time, a time encompassing all changes, appears to us in the
same manner in which ego, alter, object and space emerge in the development
of the child to produce a useful organisation of the world. The laws
of physics and absolute time emerge in the construction of physics
at the same step as a consequence of the same class of dialectic opening
that creates the terms of an opposition that produces understanding.
In this case, system-environment (not system) implies absolute time.
\begin{defn}
\emph{Events}. An event is a change in the sensible world that occurs
in a relatively short interval of time and as such is idealised as
instantaneous. We remove from the event the determinations of time
and space, which are thought of as its circumstances. Thus, events
occur at a given location and a given time.
\end{defn}

\subsection{Observers}

The observer (\emph{ego}) describes the world by measuring all distances
with respect to a point of her/his election. We use the notation $x_{A}^{a}$
for the position of body $A$ as measured by observer $a$. When facing
the need of relating positions of different observed bodies, conforming
to NAP and the search for an objective description of physical entities,
we define: 
\begin{defn}
\emph{\label{def:relpos}(relative position): }We call $x_{AB}=x_{A}^{a}-x_{B}^{a}$
the\emph{ relative position} between observed bodies $A$ and $B$,
as seen from the \emph{ego} $a$ at a given moment. 
\end{defn}
\begin{cor}
It is a demand of objectivity (by NAP) that $x_{A}^{a}-x_{B}^{a}=x_{A}^{b}-x_{B}^{b}$
(hence, there is no superindex in $x_{AB}$). 
\end{cor}
In a sense related to Leibniz, objective space is relational. Indeed,
one may suspect that vector spaces have been constructed to this end.
The group relating the arbitrary choices of different observers is
$IO(3)$, the semi-direct product of isometries and translations.

In the sequel, we will call \emph{subjective} the space created by
an arbitrary choice of reference point and base vectors orientations.
\begin{defn}
\label{def:Velocity} \emph{Relative velocity} is the rate of change
of relative position between two bodies with respect to the change
in time,
\[
v_{AB}=\frac{d(x_{A}^{a}-x_{B}^{a})}{dt}=\lim_{\Delta t\rightarrow0}\frac{\left(x_{A}^{a}(t+\Delta t)-x_{B}^{a}(t+\Delta t)\right)-\left(x_{A}^{a}(t)-x_{B}^{a}(t)\right)}{\Delta t}.
\]
\end{defn}

\subsection{Subjective velocity and Galileo transformations}

The relative velocity between the reference point chosen by the observer
$a$ and a body $A$, follows Definition \ref{def:Velocity},
\[
v_{Aa}=\frac{d(x_{A}^{a}-x_{a}^{a})}{dt}=\lim_{\Delta t\rightarrow0}\frac{\left(x_{A}^{a}(t+\Delta t)-x_{a}^{a}(t+\Delta t)\right)-\left(x_{A}^{a}(t)-x_{a}^{a}(t)\right)}{\Delta t}.
\]
We now focus on the subjective operation consisting in setting $x_{a}^{a}(t+\Delta t)-x_{a}^{a}(t)=0^{a}$
and similarly $v_{a}^{a}=\lim_{\Delta t\rightarrow0}\frac{x_{a}^{a}(t+\Delta t)-x_{a}^{a}(t)}{\Delta t}=0^{a}$,
this is to say that for the observer, the point designated as reference
by her/his arbitrary decision does not move (we have added the superscript
$a$ to the zero to indicate the subjectivity). Now the relative velocity
reads 
\[
v_{Aa}=\frac{d(x_{A}^{a}-x_{a}^{a})}{dt}=\lim_{\Delta t\rightarrow0}\frac{x_{A}^{a}(t+\Delta t)-x_{A}^{a}(t)}{\Delta t}-0^{a}\equiv v_{A}^{a}-0^{a},
\]
where we call $v_{A}^{a}$ the \emph{subjective velocity} of $A$
as established by observer $a$.
\begin{prop}
\label{prop:GT}The \emph{Galilean transformation} between observers
(reference points) $a$ and $b$ is given by,

\selectlanguage{english}%
\[
v_{A}^{a}=v_{A}^{b}+v_{ba}.
\]
 It is an operation that belongs to the group associated by NAP to
the concept of subjective velocity.
\end{prop}
\selectlanguage{english}%
\begin{proof}
We begin by writing the equality
\begin{eqnarray*}
x_{A}^{a}(t+dt)-x_{A}^{a}(t)-\left(x_{a}^{a}(t+dt)-x_{a}^{a}(t)\right) & = & x_{A}^{b}(t+dt)-x_{A}^{b}(t)-\left(x_{a}^{b}(t+dt)-x_{a}^{b}(t)\right),
\end{eqnarray*}
 which after a rearrangement reads
\begin{eqnarray*}
(x_{A}^{a}(t+dt)-x_{A}^{a}(t))-0^{a} & = & (x_{A}^{b}(t+dt)-x_{A}^{b}(t))-0^{b}+0^{b}-\left(x_{a}^{b}(t+dt)-x_{a}^{b}(t)\right).
\end{eqnarray*}
Next we observe that 
\[
0^{b}-(x_{a}^{b}(t+dt)-x_{a}^{b}(t))=(x_{b}^{b}(t+dt)-x_{a}^{b}(t+dt))-(x_{b}^{b}(t)-x_{a}^{b}(t))=x_{ba}(t+dt)-x_{ba}(t),
\]
 with $x_{ba}$an objective relative distance. We then have 
\begin{eqnarray*}
(x_{A}^{a}(t+dt)-x_{A}^{a}(t))-0^{a} & = & (x_{A}^{b}(t+dt)-x_{A}^{b}(t))+(x_{ab}(t+dt)-x_{ab}(t))-0^{b}.
\end{eqnarray*}
Dropping the zeroes, dividing by $dt$ and taking the limit $dt\to0$,
we obtain the Galilean transformation between observers. We further
notice that the limit is not a necessary step. 
\end{proof}
\selectlanguage{british}%
\begin{lem}
Galilean transformations form a group having vector addition as internal
operation.
\end{lem}
\begin{proof}
Repeated use of Proposition \ref{prop:GT} gives, 
\begin{eqnarray*}
v_{A}^{a} & = & v_{A}^{b}+v_{ba}\\
 & = & \left(v_{A}^{c}+v_{cb}\right)+v_{ba}\\
 & = & v_{A}^{c}+\left(v_{cb}+v_{ba}\right)=v_{A}^{c}+v_{ca}.
\end{eqnarray*}
\end{proof}

\section{The law of inertia}

\subsection{Laws of Nature}

The meaning of \emph{laws of Nature} deserves some examination. In
western culture before the Enlightenment, it remitted us to God's
blueprints for the universe as in the early times of Descartes, Newton
and Leibniz. After the Enlightenment, the laws of Nature must rest
on reason \citet{kant83,kant98}. The question ‘What is a law of Nature?’
has been extensively debated in contemporary philosophy of science.
Here we will not focus on what makes a statement to be a natural law.
Nevertheless, no matter how that question is answered, certain general
features of laws can be recognised. The laws of Nature correspond
to fundamental relations in situations in which a small (minimal)
portion of the universe, the \emph{system}, is considered (through
a process of idealisation) as isolated from its environment, i.e.,
the complement in the Universe of the ideally-isolated system. Such
a notion implies that the internal organisation of the system -which
the law will make explicit-, must be independent of the environment
since this is the nature of the concept we are seeking. Hence, the
law must hold with independence of the relative location of the system
with respect to the environment and shall not be affected by the background
of changes. In our perspective, there are no laws \emph{of Nature}
but rather \emph{laws for the understanding of Nature}, which themselves
are subject to the laws of reasoning and include in their ontogeny
both experience and usefulness: the object and the subject.

\subsection{The law of inertia}
\begin{defn}
An \emph{isolated body} is an idealisation consisting in extrapolating
the (short-time) motion of bodies that are perceived to be not interacting
with other bodies.

This perception could originate in the fact that when the distance
between bodies is sufficiently large or the time of the observation
is sufficiently short, interactions do not show any appreciable effect
(i.e., it is an idealisation). 
\end{defn}
An isolated body can be regarded as being alone in the universe. As
such, it defines by itself a privileged place and reference point.
However, when we deal with several ideally isolated bodies, we must
consider the problem of their changing (relative) distances (which
is a result of the same idealising process). By NAP, the description
given from the perspective of an isolated body must be equivalent
to the description given by any other isolated body inasmuch as the
particularity of the description is only that the observer is \emph{isolated}
(i.e., not influenced in its motion by other bodies). The condition
of a body as isolated is a condition of permanence, hence, when compared
with time (the measure of change) it has to be represented by a zero
derivative of its state. Therefore, we define, 
\begin{defn}
\emph{Inertial set} is the collection of isolated bodies. The \emph{inertial
class} of observers consists of those observers that use an element
of the inertial set as point of reference.
\end{defn}
When seeking a fundamental law for isolated bodies we must consider
first the possibility of giving them fixed relative positions, but
such a law contradicts our current perceptions (we recall that in
the Aristotelian physics the motion of bodies required causes, i.e.,
``forces'').
\begin{assumption}
\label{assu:The-inertial-set} There exist at least two isolated bodies
that are not at rest relative to each other.
\end{assumption}
The next possibility to be considered is:
\begin{lem}
\label{lemma:Inertia} Isolated bodies move with constant relative
velocity\@. 
\end{lem}
\begin{proof}
Consider the law of motion of an isolated body, $A$, as described
by an observer $a$ in the inertial class. By Assumption \ref{assu:The-inertial-set}
the law of motion cannot be $\frac{d}{dt}x_{Aa}=0$. The most general
law of second order is
\[
\frac{d^{2}}{dt^{2}}x_{Aa}=\alpha(x_{Aa},v_{Aa},t).
\]
However, the law must be the same for all times since there is no
privileged time, $\alpha(x_{Aa},v_{Aa},t)=\alpha(x_{Aa},v_{Aa},0)$.
Additionally, by NAP, a second observer that has selected its reference
position at a fixed distance from $a$ must produce the same law.
It follows that the law cannot depend on $x_{Aa}$ and by the same
reasoning an observer that moves with constant relative speed with
respect to $a$ must observe the same law, hence $\alpha(x_{Aa},v_{Aa},t)=\Upsilon$,
being $\Upsilon$ a constant vector. However, such constant would
indicate a particular direction in space unless $\Upsilon=0$. Hence,
we arrive to the expression 
\begin{equation}
\frac{d^{2}}{dt^{2}}x_{Aa}=0,\label{eq:inertia}
\end{equation}
that satisfies the established requisite of permanency.
\end{proof}
\begin{cor}
\label{def:SI} An \emph{inertial observer} (i.e., an element of the
inertial class) is one in which only interactions are associated to
changes in the velocities.
\end{cor}
\begin{proof}
The result is just Lemma \ref{lemma:Inertia} in negative form.
\end{proof}
\begin{thm}
Galilean transformations belong to the group of transformations mapping
the position determinations of one inertial observer to those of another
one. 
\end{thm}
\begin{proof}
By Corollary \ref{def:SI} and Lemma \ref{lemma:Inertia} two inertial
observers behave as two isolated bodies with constant (objective)
relative velocity. The subjective velocities determined by each system
are related via Proposition \ref{prop:GT}.
\end{proof}
It is important to notice that Lemma \ref{lemma:Inertia} has an equivalent
statement in terms of subjective space and time. The proof is based
on the homogeneous and isotropic character of the subjective space. 

Notice that the inertial observer is a resource of the subjective
point of view, since the objective presentation focuses only in relations.
In terms of the objective presentation if two bodies present a relative
velocity that changes with time, then both of them cannot belong in
the inertial class, i.e., at least one of them must be in interaction
with something else. It cannot be ruled out that both of them interact
and none is in the inertial class. Hence we have,
\begin{cor}
The relative acceleration $\alpha_{BA}$, of the body $B$ with respect
to an isolated body, $A$, depends only on the relative positions
and velocities of the interacting material bodies. $\alpha_{BA}\equiv\alpha_{BA}(x_{BC},v_{BC};\mbox{ interaction attributes})$,
where $C$ is some other body in interaction with $B$.
\end{cor}

\section{Interacting Bodies}

We have shown that the only law of motion compatible with isolated
bodies is the law of inertia. It is time to consider bodies with interactions.
The same law in its negative form is: accelerations (i.e., changes
in the instantaneous velocity) are the result of interactions between
bodies. The simplest form of interaction to be considered involves
just two material bodies. Consider an observer in the inertial class
(hence, it does not interact with the bodies under consideration).
From its perspective, if bodies $A$ and $B$ do not follow the inertial
law, they must be interacting, hence the minimal setup for non-inertial
motion requires two bodies that are described by the inertial (isolated)
observer, which establishes the following relations:

\begin{eqnarray}
\frac{d}{dt}v_{A}^{a} & = & \alpha_{A}^{a}\nonumber \\
\frac{d}{dt}v_{B}^{a} & = & \alpha_{B}^{a}\label{aceleracion}\\
\frac{d}{dt}(v_{A}^{a}-v_{B}^{a})=\frac{d}{dt}v_{AB} & = & \alpha_{AB},\nonumber 
\end{eqnarray}
where \foreignlanguage{english}{$\alpha_{AB}$} is the (objective)
relative acceleration of body $A$ with respect to body $B$. 
\begin{assumption}
\label{assu:additive}Accelerations are additive i.e., for any inertial
observer $a$, the acceleration produced onto a material body (A)
jointly by bodies (B) and (C), given that B and C do not influence
each other, satisfies $\alpha_{A}^{a}(B,C)=\alpha_{A}^{a}(B)+\alpha_{A}^{a}(C)$,
i.e., it is the sum of the accelerations produced by individual interactions
of A with (B) and with (C). 
\end{assumption}
Following Newton,
\begin{assumption}
We consider in the sequel interactions of instantaneous action at
a distance.
\end{assumption}
Other types of interactions may be imagined, but we are focusing on
Newton's programme. This assumption was suggested by observation.
For gravitation, it has so far resisted refutation. Indeed, several
experimental measurements indicate that in classical terms the ``speed
of gravity'' must be larger than $10^{8}C$ (being $C$ the speed
of light), the first argument going back to Laplace \citep{flan98}.

\subsection{Gravitation and mass}

The fundamental interaction between two bodies, the one that is always
present and cannot be compensated, is called \emph{gravitation} of
matter. A fundamental ingredient at this point is Galileo's proof
that all bodies experience the same acceleration in vertical motion
\citet[Third day, p. 173--]{gali14}. This achievement rests both
on the logical need of investigating accelerated bodies after inquiring
about non-accelerated bodies (eq. \ref{eq:inertia}) and on experimental
observation.
\begin{assumption}
\label{assu:AddMass}The gravitational acceleration $\alpha_{A}^{a}(B)$
produced onto a material body ($A$) by a given one ($B$) is proportional
to a characteristic of the body $B$ named the \textbf{mass} (all
other circumstances being identical). The mass of body $B$ is an
objective property (i.e., the same, up to choice of units, for any
inertial observer).
\end{assumption}
In mathematical terms, the accelerations are, letting $g_{A}^{a}(B)$
summarise the rest of the dependencies, 

\begin{eqnarray*}
\alpha_{A}^{a}(B) & = & m_{B}g_{A}^{a}(B)\\
\alpha_{B}^{a}(A) & = & m_{A}g_{B}^{a}(A).
\end{eqnarray*}

\begin{cor}
The mass of an aggregation of matter is the aggregated mass of the
parts.
\end{cor}
\begin{notation}
From here on, we will drop the index of ego, since we are dealing
with just one observer belonging to the inertial class. For example,
we write $\alpha_{A}\equiv\alpha_{A}^{a}$. However, quantities with
two ``body''-subindices, such as $f_{AB}$ are always objective.\footnote{Binkoski \citep{bink16} argues that with three relational hypotheses
it is possible to distinguish inertial systems from non inertial,
while \textquotedbl Galilean relational space-time is too weak of
a structure to support a relational interpretation of classical mechanics.\textquotedbl{}
At this point of our formulation it is clear that Newtonian mechanics
considers an (external) observer, something what is necessarily non-relational.
In more general terms, the dialectic objective-subjective (attaining
objectivity by negating, going across, subjectivity) that we have
presented solves the apparent opposition between Newton and Leibniz.
This solution is reached by means of a critical reflection, not by
adding \emph{ad-hoc} postulates aiming to obtain some desired relation
(what would amount to resort to instrumental reason). } 
\end{notation}
\begin{assumption}
\label{assu:DepPos}The gravitational interaction depends only on
relative position.
\end{assumption}
With this we mean that letting $x_{AB}=x_{A}-x_{B}$, then $g_{A}(B)=x_{AB}\phi$,
where $\text{\ensuremath{\phi}}$ is a scalar function that may depend
on the length of the vector $x_{AB}$. When considering the gravitational
interaction of two bodies, $A$, localised in $x_{A}$ and, $B$,
localised in $x_{B}$ according to an inertial observer, we notice
that exchanging attributes between bodies is the same as exchanging
positions, hence an immediate corollary of this assumption is that
$g_{A}(B)=-g_{B}(A)$. As a consequence, to render this symmetry explicit
rather than acceleration it is convenient to consider 
\begin{defn}
\emph{\label{def:GForce}} $F_{AB}=m_{A}\alpha_{A}(B)=m_{A}m_{B}g_{A}(B)$\emph{
is called }gravitational force\emph{.} 
\end{defn}
A related quantity is
\begin{defn}
The linear momentum of a material body is the product of its mass
times its velocity $p=mv$. 
\end{defn}
Thus, the gravitational interaction is described by the symmetric
form 

\begin{eqnarray*}
\frac{d}{dt}p_{A}=m_{A}\frac{d}{dt}v_{A} & = & m_{A}\alpha_{A}(B)\equiv F_{AB}=m_{A}m_{B}g_{A}(B)\\
\frac{d}{dt}p_{B}=m_{B}\frac{d}{dt}v_{B} & = & m_{B}\alpha_{B}(A)\equiv F_{BA}=m_{B}m_{A}g_{B}(A),
\end{eqnarray*}
the latter being a definition of the concept of gravitational force
and of force, the cause for the changes in momentum.
\begin{cor}
The total momentum $p_{A}+p_{B}$ is constant in time when the bodies
interact gravitationally, i.e.,
\[
\frac{d}{dt}\left(p_{A}+p_{B}\right)=\frac{d}{dt}p_{A}+\frac{d}{dt}p_{B}=m_{A}m_{B}g_{A}(B)+m_{B}m_{A}g_{B}(A)=0.
\]
\end{cor}
The previous corollary is an immediate consequence of Assumption \ref{assu:DepPos},
i.e., of $g_{A}(B)=-g_{B}(A)$. In Classical Mechanics it is called
\emph{Newton's third law} and also \emph{Principle of Action and Reaction}.

\subsection{Other interactions}

Let us now consider other interactions following the previous Newtonian
scheme. Assumptions \ref{assu:AddMass} and \ref{assu:DepPos} are
too narrow to encompass other types of interactions, despite its adequacy
for gravitation. We rephrase them in a more general setting as 
\begin{assumption*}
\textbf{\textup{{[}5$^{\prime}${]}}} The defining property $Q_{i}$
is additive for $f$, i.e.,{\small  
\[
f(Q_{1}+Q_{2},x_{A},\dots,Q_{B},x_{B},\dots)=f(Q_{1},x_{A},\dots,Q_{B},x_{B},\dots)+f(Q_{2},x_{A},\dots,Q_{B},x_{B},\dots).
\]
}
\end{assumption*}
Here $Q_{i}$ denotes the defining property of body $i$ associated
to the interaction, i.e., the necessary information in order to determine
the interaction, and $f$ is the associated force, i.e., a quantity
representing the departure of bodies $A$ and $B$ from inertial motion.
For example, $Q$ stands for electric charge and $f$ for the Coulomb
electrostatic force (see the final paragraph of Subsection \ref{subsec:tension}
for further discussion).
\begin{assumption*}
\label{do not touch aceleration} \textbf{\textup{{[}6$^{\prime}${]}}}
We consider in the sequel interactions that depend only on relative
position and relative velocity.
\end{assumption*}
Newtonian mechanics is based on the generalisation of Definition \ref{def:GForce}
to other types of interactions, namely
\begin{assumption}
\label{assu:F=00003Dma}Force $f$ is the cause of acceleration. The
latter is proportional to the ratio between force and mass.
\end{assumption}
We have already realised that a consequence of Lemma \ref{lemma:Inertia}
is that the presence of accelerations is an indication (or a symptom)
of the existence of interactions. This assumption states that the
totality of the intervening interactions is \emph{exhausted} in the
acceleration, they are its \emph{cause.} We should notice that this
assumption must be contrasted with experiments\footnote{For example, Weber's formalisation of Faraday's induction law \citet{webe46}
requires for the induction (electromotive) force to depend on the
acceleration (as Faraday's law also does). This may be interpreted
as having a mass that depends on the relative state of motion of the
interacting charges.}. Hence, we write,

\begin{eqnarray}
m_{A}\frac{d^{2}}{dt^{2}}x_{A} & = & f(Q_{A},x_{A},\dots,Q_{B},x_{B},\dots)\label{eq:NE2G}\\
m_{B}\frac{d^{2}}{dt^{2}}x_{B} & = & k(Q_{A},x_{A},\dots,Q_{B},x_{B},\dots).\label{eq:NE2G2}
\end{eqnarray}
In this presentation there is no need to introduce \emph{inertial
mass} as something different from \emph{gravitatory mass}. Gravitation
is taken to be a fundamental (unavoidable) interaction and its associated
mass is what enters in eqs (\ref{eq:NE2G} and \ref{eq:NE2G2}). This
is a consequence of the symmetry of the presentation, by which whatever
pertains the gravitational interaction is explicitly displayed.

\subsubsection{Consequences\label{subsec:Consecuencias}}

The above symmetry considerations may be propounded for eqs. (\ref{eq:NE2G}
and \ref{eq:NE2G2}). In fact, $f(Q_{A},x_{A},\dots,Q_{B},x_{B},\dots)$
is the force that the body with property $Q_{B}$ and geometrical
parameters $x_{B},\cdots$ exerts on the body with property $Q_{A}$
(and $x_{A},\cdots$). The dots indicate that forces may depend on
other geometrical properties than just position (velocity, for example).
The quantity $k(Q_{A},x_{A},\dots,Q_{B},x_{B},\dots)$ is the corresponding
force that $A$ exerts on $B$. It follows by NAP that 
\begin{equation}
k(Q_{A},x_{A},\dots,Q_{B},x_{B},\dots)=f(Q_{B},x_{B},\dots,Q_{A},x_{A},\dots).\label{ojo}
\end{equation}

\begin{lem}
\label{lem:AyR} Under the previous assumptions, for a pair of bodies
$A$ and $B$ in interaction, the forces $f$ (force on $A$ caused
by $B)$ and $k$ (force on B caused by $A)$ satisfy the Generalised
Principle of Action and Reaction, namely 
\begin{eqnarray*}
f & = & x_{AB}\phi_{e}+v_{AB}\phi_{s}+\left(x_{AB}\times v_{AB}\right)\phi_{\bot}\\
k & = & -x_{AB}\phi_{e}-v_{AB}\phi_{s}+\left(x_{AB}\times v_{AB}\right)\phi_{\bot},
\end{eqnarray*}
where $\phi_{e},\phi_{s},\phi_{\bot}$ are scalar functions of relative
position, relative velocity and possibly other parameters, $x_{AB}=x_{A}-x_{B}$
and $v_{AB}=v_{A}-v_{B}$.
\end{lem}
\begin{proof}
By Assumption $6^{\prime}$ there is no other geometric dependency
and we can replace the occurrences of $x_{A},x_{B}$ by $x_{AB}$
and those of $v_{A},v_{B}$ by $v_{AB}$. Let $T$ denote the operation
of exchanging the geometrical properties associated to bodies $A$
and $B$ (or, what is the same, switching their defining properties
$Q_{i}$). We have, by NAP,

\begin{eqnarray}
Tf(Q_{1},Q_{2},x_{AB},v_{AB},\cdots,t) & = & f(Q_{2},Q_{1},x_{AB},v_{AB},\cdots,t)\nonumber \\
 & = & k(Q_{1},Q_{2},-x_{AB},-v_{AB},\cdots,t)\label{intercambio}\\
Tk(Q_{1},Q_{2},-x_{AB},-v_{AB},\cdots,t) & = & k(Q_{\text{2}},Q_{\text{1}},-x_{AB},-v_{AB},\cdots,t)\nonumber \\
 & = & f(Q_{1},Q_{2},x_{AB},v_{AB},\cdots,t)\nonumber 
\end{eqnarray}
and consequently $T^{2}=Id$. The most general expression for a force
between two interacting bodies depending as vector only on the vectors
$x_{AB}$ and $v_{AB}$ is 
\begin{eqnarray*}
f(Q_{1},Q_{2},x_{AB},v_{AB},t) & = & x_{AB}\phi_{e}+v_{AB}\phi_{s}+\left(x_{AB}\times v_{AB}\right)\phi_{\bot}\\
k(Q_{1},Q_{2},x_{AB},v_{AB},t) & = & x_{AB}\psi_{e}+v_{AB}\psi_{s}+\left(x_{AB}\times v_{AB}\right)\psi_{\bot},
\end{eqnarray*}
where the multiplicative factors $\phi_{e},\phi_{s},\phi_{\bot},\psi_{e},\psi_{s},\psi_{\bot}$
are scalar functions of relative position, relative velocity and possibly
other parameters. The reason is that relative position and relative
velocity define only three vectors in space. In other words, $\phi_{e}$,
etc., may depend on various properties (charge, velocity, etc.) but
are scalar coefficients. The vector property of $f$ and $k$ is given
by linear combinations of $x_{AB}$, $v_{AB}$ and $x_{AB}\times v_{AB}$.
We recall that the vector $w\times z$ (called \emph{cross product}
or \emph{vector product} of $w$ and $z$) has a direction that is
perpendicular to both $w$ and $z$ and whenever $w$ and $z$ are
collinear, $w\times z=0$.

Exchanging the properties $Q_{1}$ and $Q_{2}$ of $A$ and $B$ is
equivalent to keeping the properties and exchanging their masses,
positions and velocities. Again by NAP, from eqs.(\ref{ojo} and \ref{intercambio})
we have,
\begin{eqnarray}
f(Q_{1},Q_{2},x_{AB},v_{AB},t) & = & x_{AB}\phi_{e}+v_{AB}\phi_{s}+\left(x_{AB}\times v_{AB}\right)\phi_{\bot}\nonumber \\
 & = & -x_{AB}\psi_{e}-v_{AB}\psi_{s}+\left(\left(-x_{AB}\right)\times\left(-v_{AB}\right)\right)\psi_{\bot}\label{eq:AyR}\\
 & = & -x_{AB}\psi_{e}-v_{AB}\psi_{s}+\left(x_{AB}\times v_{AB}\right)\psi_{\bot},\nonumber 
\end{eqnarray}
which produces 
\begin{equation}
\phi_{e}=-\psi_{e},\,\phi_{s}=-\psi_{s},\,\phi_{\perp}=\psi_{\perp}.\label{accion y reaccion}
\end{equation}
\end{proof}
\begin{rem}
Gravitation, as conceived by Newton, has vanishing $\phi_{s}$ and
$\phi_{\bot}$. 
\end{rem}
\begin{lem}
\label{cor:GAyR}Under the Generalized Principle of Action and Reaction
the total momentum, $p_{A}+p_{B}$, is constant in time if and only
if $\phi_{\perp}=0$.
\end{lem}
\begin{proof}
We just compute 
\begin{eqnarray*}
\frac{d}{dt}\left(p_{A}+p_{B}\right) & = & f+k\\
 & = & x_{AB}(\phi_{e}-\phi_{e})+v_{AB}(\phi_{s}-\phi_{s})\\
 &  & +\left(x_{AB}\times v_{AB}\right)(\phi_{\bot}+\phi_{\bot})\\
 & = & 2\left(x_{AB}\times v_{AB}\right)\phi_{\bot}.
\end{eqnarray*}
Moreover, if $(x-y)$ is parallel to $(u-v)$ (unidimensional relative
motion) the actual value of $\phi_{\perp}$ is irrelevant and we can
choose to set $\phi_{\perp}=0$ also in that situation. 
\end{proof}
We must notice that the total momentum as perceived by the observer
is a conserved quantity if and only if $\phi_{\perp}=0$ for all forces
in classical mechanics. The idea of conservation of total momentum
is a consequence of Newtonian tradition and not a demand of reason.
It must be established in every new theory of Nature as an additional
assumption subject to empirical consideration.
\begin{lem}
Under the Generalized Principle of Action and Reaction there is no
internal torque if and only if $\phi_{s}=\phi_{\perp}=0$.
\end{lem}
\begin{proof}
Following Corollary \ref{cor:GAyR}, we compute, recalling that $\frac{d}{dt}v_{AB}=\alpha_{AB}=\frac{f}{m_{A}}-\frac{k}{m_{B}}$,
\begin{eqnarray*}
x_{AB}\times\frac{m_{A}m_{B}}{m_{A}+m_{B}}\frac{d}{dt}v_{AB} & = & x_{AB}\times\frac{m_{A}m_{B}}{m_{A}+m_{B}}\left(\frac{f}{m_{A}}-\frac{k}{m_{B}}\right)\\
 & = & x_{AB}\times\left(x_{AB}\phi_{e}+v_{AB}\phi_{s}\right)\\
 &  & +\left(\frac{m_{B}-m_{A}}{m_{A}+m_{B}}\right)x_{AB}\times\left(x_{AB}\times v_{AB}\right)\phi_{\bot}.
\end{eqnarray*}
The dependence on $\phi_{e}$ vanishes by the properties of the cross
product. Since the remaining terms in the rhs have contributions in
different directions, 
\[
\frac{d}{dt}\left[x_{AB}\times\frac{m_{A}m_{B}}{m_{A}+m_{B}}v_{AB}\right]=0\Leftrightarrow\phi_{s}=\phi_{\perp}=0.
\]
 In Classical Mechanics the quantity in square brackets is called
\emph{angular momentum} and its time-derivative is called \emph{torque}.
\end{proof}

\subsection{Central Forces}

An important issue in Classical Mechanics following the Newtonian
tradition is that of central forces depending only on the relative
distance. From the previous results, the following particular case
can be highlighted: 
\begin{lem}
\label{lem:EnergyCons}Under the previous assumptions, if $\phi_{s}=0$
and $\phi_{e}=h(|x_{AB}|)$ (a scalar function of relative distance
only) then there exists a dynamical quantity (called internal energy)
that is constant in time.
\end{lem}
\begin{proof}
A standard computation using the relative (objective) quantities defined
previously yields,
\[
\left(v_{AB}\right)\cdot\frac{d}{dt}\left(v_{AB}\right)\equiv\frac{1}{2}\frac{d}{dt}\left(v_{AB}\right)^{2}=\left(v_{AB}\right)\cdot\left(\frac{f}{m_{A}}-\frac{k}{m_{B}}\right),
\]
which by eqs. (\ref{eq:AyR} and \ref{accion y reaccion}) reads (since
$v_{AB}\cdot\left(x_{AB}\times v_{AB}\right)=0$), 
\[
\frac{1}{2}\frac{d}{dt}\left(v_{AB}\right)^{2}=\left(\frac{1}{m_{A}}+\frac{1}{m_{B}}\right)v_{AB}\cdot\left(x_{AB}\phi_{e}+v_{AB}\phi_{s}\right).
\]
In other words, under the conditions of this Lemma, we have, 
\begin{eqnarray*}
\frac{1}{2}\frac{m_{A}m_{B}}{m_{A}+m_{B}}\frac{d}{dt}\left(v_{AB}\right)^{2} & = & h(|x_{AB}|)\,x_{AB}\cdot v_{AB}.
\end{eqnarray*}
We identify in the lhs the \emph{kinetic energy }for the relative
motion. Let $-V(|x_{AB}|)$ be such that $-\nabla V=h(|x_{AB}|)\,x_{AB}$.
Then, $-\frac{dV}{dt}=h(|x_{AB}|)\,x_{AB}\cdot v_{AB}$ and therefore

\[
\frac{d}{dt}\left(\frac{1}{2}\frac{m_{A}m_{B}}{m_{A}+m_{B}}\left(v_{AB}\right)^{2}+V(|x_{AB}|)\right)=0.
\]
The quantity $E=\frac{1}{2}\frac{m_{A}m_{B}}{m_{A}+m_{B}}\left(v_{AB}\right)^{2}+V(|x_{AB}|)$
is known in Classical Mechanics as \emph{internal energy.}
\end{proof}

\subsection{The tension between inertial and non-inertial\label{subsec:tension}}

One of the pillars of our understanding of Nature is observation,
which invariably takes place on Earth (or since recently on its neighbouring
galactic surroundings). Both the surface of the Earth and the Solar
System are regarded as non-inertial references by Classical Mechanics.
Indeed, we have no reasons to doubt of this non-inertiality since
we detect its effects. However, since the Leibniz-Clarke discussion
and lately Mach (see below) it has been an important issue the fact
we have no means to decide whether a reference system is inertial
(fully free from interactions) or not.

Despite this issue, physics has succeeded in conceiving different
fundamental interactions, such as gravitation or the electrostatic
interaction as described by Coulomb. One may wonder how this could
be done in the first place, without having a clue about ``how inertial''
we are. The solution presented here finds its support on three concepts:
(a) the process of idealisation described by Galileo, by which one
identifies and eliminates from observation what is conceived to be
foreign to the interaction under study, e.g., the effect of friction
forces, or the presence of other interactions influencing both the
system and the observer, (b) a fully-objective approach where the
description of an interaction is performed using objective (invariant)
quantities belonging to the interaction pair (relative positions,
charges and the like) and (c) the honest effort to keep the description
free from influences coming from the observer to the largest possible
extent (this description is always provisional until a new, so-far
neglected, influence of the observer is detected). We return to this
issue in Subsection \ref{subsec:present}.

It is illustrative to consider the way in which the founders of electromagnetism
studied the new interactions, an approach based on Assumption \ref{assu:additive}.
 Prior to experimenting with the new interaction, they let the system
to be at rest in the observer's frame as a result of the equilibrium
of the (other) present forces and accelerations. Next, the new influential
condition was established and when a new equilibrium was reached they
measured the balancing force judged to be equivalent (but opposite
in sign) to the newly introduced force. Thus, forces are introduced
as members of an equivalence class rather than by a direct application
of the definition, i.e., no acceleration is truly measured in a first
instance. Such were the methods of Coulomb and Ampère for example.

\section{On Newton's laws of Classical Mechanics}

Those, like us, that have tried to reach a deeper understanding of
Newton's ideas by reading the \emph{Principia}, might have come to
the conclusion that Newton left behind little or no clues on the fundamentals
behind his axioms. In the Scholium to the Definitions \citet{newt87},
Newton writes his famous notion of absolute space and time. Although
he recognises them as related to the common intuition of ``the vulgar''
he proposes that the vulgar conceptions are imperfect images of absolute
time and absolute space, a position that reminds us of the precedence
of Platonic worlds, being the world of ideas the real world \citet{plato60BC}.
Needless to say, Newton relies on the vulgar notions to deliver and
argument for his notions of space and time, place and motion. He indeed
introduces without explicit recognition the idea of Galilean transformations:
``Thus in a ship under sail, the relative place of a body is that
part of the ship which the body possesses, {[}...{]} But real, absolute
rest, is the continuance of the body in the same part of that immovable
space in which the ship itself, its cavity, and all that it contains
is moved. Wherefore if the earth is really at rest, the body which
relatively rests in the ship, will really and absolutely move and
with the same velocity which the ship has on the earth. But if the
earth also moves the truly and absolute motion of the body will arise,
partly from the true motion of the earth in immovable space; partly
from the relative motion of the ship on the earth; and if the body
moves also relatively in the ship; its true motion will arise, partly
from the true motion of the earth in immovable space, and partly from
the relative motions as well of the ship on the earth, as of the body
on the ship;...'' \citep[(Motte)][p. 78]{newt87} For Newton, it
is relative space what can be conceived by the vulgar as fixed to
the heavens: ``All things are placed in time as to order of succession;
and in space as to order of situation. It is from their essence or
nature that they are places; and that the primary place of things
should be movable, is absurd {[}...{]} But because the parts of space
cannot be seen, or distinguished from one another by our senses, therefore
in their stead we use sensible measures of them. For from the positions
and distances of things from any body considered as immovable, we
define all places; and then with respect to such places we estimate
all motions... And so, instead of absolute places and motions we use
relative ones, and that without any inconvenience in common affairs;
but in philosophical disquisitions, we ought to abstract from our
senses and consider things themselves, distinct from what are only
sensible measures of them. For it may be that there is no body at
rest to which the places and motions of others may be referred \citep[(Motte)][p. 79]{newt87}
(Cajori corrected the last phrase to: ``And therefore, as it is possible,
that in the remote region of fixed stars, and perhaps far beyond them,
there may be some body absolutely at rest...'' \citep[(Mote-Cajori)][]{newt87}).

Newton's construction cannot come to an end unless the existence of
an end, absolute space, is introduced \emph{ad-hoc}. In terms of the
present work, Newton's construction is moving from the subjectivity
of one reference into the subjectivity of another, each one carrying
its relative space. The sequence of subjective views can only finish
in an objective view. Is God the observer of absolute space? Thus,
rather than fixed to the heavens, absolute space is fixed to the Heavens,
i.e., objective space would be God's perspective. The same sort of
recursion is present in absolute time (see General Scholium, \citet[p. 503--]{newt87}).

Concerning time, Newton did not conceive a method of measuring absolute
time, but only approximations to the ideal time: ``Absolute time,
in astronomy, is distinguished from relative, by the equation or correction
of the vulgar time. For the natural days are truly unequal, though
they are commonly considered as equal, and used for a measure of time;
astronomers correct this inequality for their more accurate deducing
of the celestial motions. It may be, that there is no such thing as
an equable motion, whereby time may be accurately measured. All motions
may be accelerated and retarded; but the true, or equable, progress
of absolute time is liable to no change.`` \citep[(Motte)][p. 78]{newt87}

Newton conceived absolute space and time as the limit of a process
in which relative motions were added and clocks were perfected in
their regularity. The opposite view of space in those years was that
of Leibniz who sustained the relational view. We have shown in this
work that both views are intimately related and that absolute space
and time consist in a final attempt to maintain the observer in the
scene. In our discussion, in contrast, subjective space and motion
are the result of the introduction of the idealised observer in the
scene. Newton's view of absolute time and space corresponds to the
egocentric stage of the child, as observed by Piaget \citep[p. 367]{piag99},
later to be modified by the recognition of the other. It is important
to notice that although relative space appears as objective, relative
time appears only as intersubjective, being impossible for us to define
an objective measure of time without an \emph{ad-hoc} assumption,
which in all cases amounts to construct a physical equivalent of our
memory such as logging the tics of a clock or the oscillations of
a pendulum.

In \emph{Principia Mathematica} Newton founded the concept of mass
on the intuition of weight: ``The quantity of matter is the measure
of the same, arising from its density and bulk conjointly {[}...{]}
It is this quantity that I mean hereafter everywhere under the name
of body or mass. And the same is known by the weight of each body;
for it is proportional to the weight, as I have found by experiments
on pendulums, very accurately made, which shall be shown hereafter.''
\citep[(Motte)][]{newt87} Such an idea was later challenged in an
\emph{a-posteriori} empiricist view by Mach \citet{mach19}. We have
shown in this work that Newton's perspective is proper to the construction
of knowledge, this is, the genetic meaning of mass emerges from interactions.
We realise that the gravitational interaction is the most notorious
one, and pre-exists the concept of force. In the terms of Popper,
Newton's theory is simpler than Mach's proposal, since it lends itself
more generously to refutation \citet{popp59}.

It is worth noting that \emph{Principia Mathematica} dedicates special
attention to the issue of relative vs. absolute rotations. A fundamental
difference between rotational and translational motion is that rotations
cannot be conceived without acceleration. In the Scholium to the first
chapter Newton discusses centripetal forces: ``\emph{A centripetal
force is that by which bodies are drawn or impelled, or in any way
tend, towards a point as to a centre}. Of this sort is gravity''
\citep[(Motte)][p. 74]{newt87} and he continues later ``The effects
which distinguish absolute from relative motion are, the forces of
receding from the axis of circular motion. For there are no such forces
in a circular motion purely relative, but in a true and absolute circular
motion, they are greater or less, according to the quantity of the
motion.'' \citep[(Motte)][p. 80-81]{newt87} This view contrasts
with Mach's attempts to wipe out the difference between relative and
absolute rotations: ``The principles of mechanics can, indeed, be
so conceived, that even for relative rotations centrifugal forces
arise.'' \citet[(II.VI.5, p. 232)]{mach19}. In that paragraph Mach
remains within the kinematic aspect of relative rotations; what we
may call the visual intuition of rotation. However, in the construction
of knowledge, the child gradually adds other sensations to the visual,
noticeably the physical effort required to sustain a rotation. The
distinction relates to our experiences as children, with relative
motion related to spinning around our vertical axis (while standing
up) and absolute circular motion such as the one used to throw stones
with a sling. While the visual appearance is equivalent, the complete
sensorial experience is quite different. The sensorial difference
has to be accounted for and Newton introduces the distinction between
absolute and relative rotation. Mach makes an enormous metaphysical
effort -thus betraying his own ideas- to disregard the difference
in sensory effects arising in actual rotations as opposed to apparent
ones (called in this context ``relative'' both by Newton and Mach).
Poincaré \citet[(Ch. VII)]{poin13} criticised Mach's view accurately
indicating that in order to confuse apparent rotations with true ones,
we need to assign to the apparent rotation a force which is contrary
to intuition (it increases with distance), a sort of conspiracy of
the rest of the universe to deceive the observer. Mach's attempt to
suppress the subject comes short of Newton's insight about the difference
between subjective and objective descriptions.

Poincaré \citep{poin13} insisted on this matter proposing for discussion
an ``Earth'' always covered (clouded), so that celestial references
could not be observed. The scientists of this ``Earth'' would discuss
relative and absolute rotations. ``In the theory of relative motion
we observe, besides real forces, two imaginary forces, which we call
ordinary centrifugal force and compounded centrifugal force. Our imaginary
scientists can thus explain everything by looking upon these two forces
as real, and they would not see in this a contradiction of the generalised
principle of inertia, for these forces would depend, the one on the
relative positions of the different parts of the system, such as real
attractions, and the other on their relative velocities, as in the
case of real frictions. Many difficulties, however, would before long
awaken their attention. If they succeeded in realising an isolated
system, the centre of gravity of this system would not have an approximately
rectilinear path. They could invoke, to explain this fact, the centrifugal
forces which they would regard as real, and which, no doubt, they
would attribute to the mutual actions of the bodies—only they would
not see these forces vanish at great distances— that is to say, in
proportion as the isolation is better realised. Far from it. Centrifugal
force increases indefinitely with distance. Already this difficulty
would seem to them sufficiently serious, but it would not detain them
for long. They would soon imagine some very subtle medium analogous
to our ether, in which all bodies would be bathed, and which would
exercise on them a repulsive action. But that is not all. Space is
symmetrical—yet the laws of motion would present no symmetry. They
should be able to distinguish between right and left. They would see,
for instance, that cyclones always turn in the same direction, while
for reasons of symmetry they should turn indifferently in any direction.
If our scientists were able by dint of much hard work to make their
universe perfectly symmetrical, this symmetry would not subsist, although
there is no apparent reason why it should be disturbed in one direction
more than in another. They would extract this from the situation no
doubt they would invent something which would not be more extraordinary
than the glass spheres of Ptolemy...'' The discussion continues for
long, mathematising this ``Earth'' until reaching a contradiction.

In terms of genetic epistemology a central place is taken by \emph{cognitive
surpass}\footnote{Original in Spanish: \emph{rebasamiento cognoscitivo} \citep{piag89}.}
shortly said, it consists in a reordering of previous knowledge, extending
it into more general (universal) forms. The old knowledge, however,
is not suppressed. It becomes rather a particular form of the new
one, a form that supports as well the new structure of knowledge.
Thus, the distinction about rotating around our axis and that of revolving
a stone with a sling must be preserved. As indicated by Poincaré,
Mach's posit destroys this difference. Still, the difference does
not imply absolute space, because the latter means to conceive rotations
with respect to something else, while for an extended body like the
Earth rotations correspond to the movement of one part of the Earth
around another. Hence, it is possible and simple to forget about the
stars still preserving the eidetic seeing that originated the distinction.
In the relationalist point of view, it can be said as well that the
stone revolves around us, or both revolve around the common centre
of mass, but in all cases, sensible forces are exerted between the
revolving bodies, the bodies interact. The same can be said with respect
to (parts of) the Earth, but the heavens have nothing to do with it. 

\subsection{The present view\label{subsec:present}}

In this essay we explore the consequences of recognising our humanity,
the acceptance of the undeniable fact that our thoughts have their
grounds in the constructions we produced as babies and infants, later
integrated with our cultural background. This structure is built always
under the supervision of reason, since reason protects the unity of
our conception of the world: the construction by dialectical oppositions
that we have come to call understanding. Thus, the starting point
of our construction is a principle of knowledge, an internal requirement
of both what we are ready to accept as knowledge of the world and
what has to be rejected as such. We named this principle the No Arbitrariness
Principle (NAP). On the positive side, it enables us to pursue the
quest of developing an objective knowledge of the world. Since knowledge
means to put the input of the sensory system in correspondence with
the organisational labour of our brain\footnote{This point of view has several precedents. Kant, \citet[ Introduction]{kant87}
opens the book by saying ``That all our knowledge begins with experience
there can be no doubt. For how is it possible that the faculty of
cognition should be awakened into exercise otherwise than by means
of objects which affect our senses, and partly of themselves produce
representations, partly rouse our powers of understanding into activity,
to compare to connect, or to separate these, and so to convert the
raw material of our sensuous impressions into a knowledge of objects,
which is called experience?'' and later ``Understanding cannot intuite,
and the sensuous faculty cannot think. In no other way than from the
united operation of both, can knowledge arise''. Husserl formulates
similar ideas as ``Isolated experience, even when it is accumulated,
is still worth little to it. It is in the methodical disposition and
connection of experiences, in the interplay of experience and thought,
which has its logically rigid laws, that valid experience is distinguished
from invalid, that each experience is accorded its level of validity,
and that objectively valid knowledge as such, knowledge of nature,
is worked out.'' \citet[p. 234]{huss56} Piaget and Garcia state
``Not only is there an absence of a clear frontier between the contributions
of the subject and those of the object (since we can only know about
interactions between the two) but in addition, it is only to the extent
that logical and mathematical structures are applied that one can
come to attain the object, and objectivity improves as a function
of richer logico-mathematical structures. In fact, the elementary
perceptual object is already partially 'logicized' from the start,
even though much less 'objective' than the more elaborated object.''
\citet[Introduction]{piag89}}, truly objective knowledge can only be sought by a second motion
in which the subject (the observer) removes her/himself from the scene.
This is the master plan of this work, attaining as much objectivity
as possible by going across (transcending) subjectivity.

If space and time have their foundation in the early experiences of
life, interactions are a rather different matter. Interactions are
not objects in space. Forces are metaphysical entities responsible
for the departure from inertial motion, they have no time since causes
are only made apparent by their effects: the gravitational interaction
is present before, during and after the fall of an object. Forces
can only be recognised by their effects and characterised using the
established framework in a process of inference. The first law of
motion has then a stronger support that the other two; the result
is independent of the way we characterise interactions. The second
law and even more the third are in a good degree the result of promissory
assumptions. Newton, in search of credibility for the principle of
action and reaction, relies on contact forces that impress us in a
sensory form (``If you press a stone with your finger, the finger
is also pressed by the stone'' \citep[(Motte)][]{newt87}). In our
discussion we explore to what extent the principle of action and reaction
is a result of NAP, restricting our discussion to forces of instantaneous
action. The result is a generalised principle of action and reaction.
Thus, Newton's third law has a weaker logical support than the first
and second laws. If classical forces are ever found that break reflection
symmetry, we expect total linear momentum not to be conserved by the
interaction.

Apart from assumptions of technical character and others suggested
by observation, the present approach rests on a principle of knowledge
(derived from NAP), namely that the relative properties of a pair
of interacting objects depend only on the objects themselves and their
interaction. Thus, we develop the concept of objective relative distance.
Indeed, the construction rests on the resolution of the tension between
objective and subjective. In this sense, we get past the traditional
issue present in Leibniz, Newton and later Mach and Poincaré (see
below) about ``absolute'' and ``relative'' distances. The quest
here is to objectivise relative distances. The present work disposes
of absolute space, while retaining objectivity. Another basic underlying
idea is that the concepts of space and time are \emph{different}.
Space is related to permanence and time to change. In particular,
space and time are not interchangeable. This is based in the way we
construct our knowledge. 

\section{On the laws of Nature}

In a society that rests on religious concepts it makes sense to consider
that Nature is a creation of God and that She/He endowed Nature with
laws. As religious beliefs decline and faith in our own strength grows,
we come to accept that the laws of Nature are actually laws produced
by humanity to understand Nature, they are \emph{laws for the understanding
of Nature}, constructs generated to organise the sensory world, the
continuation at the level of civilisation of the work autonomously
initiated by every child. As such, the laws of Nature are subject
to a higher set of rules, those discovered by every member of our
species in our early quest for survival. They include the conviction
that ``there is something out there'' that reaches us through our
senses (call it Nature) and therefore, the possibility of objective
knowledge.

Unlike social rules, on which we agree in order to preserve energies
for other relevant social matters and are established by consensus
or power, unlike the ``agreement on disagreeing'' that makes room
for cooperation by setting aside disputes, Truth is the only form
of agreement with Nature, as Nature is not a social actor that can
change behaviour or negotiate the rules. But Truth requires an agreement
with our humanity as well. Thus, the laws of Nature or \emph{truth
about Nature} come as a result of the dialectic interplay between
a universal (humanwise - or civilisationwise) subject and the universe
of sensory observations, Nature. None of them can be absent in the
laws of Nature. In turn, the universal subject requires a correspondence
between the individuals, it requires intersubjectivity. Intersubjectivity
sets the lower level of requisites on the side of the subject for
calling something a law of Nature.

Much of our understanding of Nature comes from what we call ``dialectical
openings to knowledge''. Thus, the recognition of ourselves requires
the simultaneous recognition of not-ourselves, our environment. Likewise,
the motion of isolated bodies, the inertial motion, has its necessary
opponent in the motion of bodies that interact. At first sight, we
could have introduced just unilateral action since the opponent of
isolated is not-isolated, i.e., influenced by (or influencing) others.
But such an idea contradicts the higher level of reasoning, the rules
on rules, one of them being the No Arbitrariness Principle (NAP).
Do we have a reason to support that one body can influence another
without being influenced, that there is an asymmetry among bodies?
So far, the answer is no. 

The present approach is not just discursive, it is constructive as
well. Let us see how it works on a much debated matter. 

\subsection{The speed of light}

During a large part of the XIX-th century and well into the following
century, a number of attempts to measure the speed of light and to
understand its constancy were done. These attempts were mostly independent
of the electromagnetic theory of light (at most, the wave-like properties
of light were used). Experiments such as those by Fizeau in 1848,
later developed and improved by Cornu in 1872-74 \citet{corn76} (using
a light source/detector, a rotating cogwheel and a mirror at rest
relative to the source/detector) consider that there exists a light
path and that the event ``detect the light'' occurs at a later time
than the event ``turn on the light''. The measuring issue can hence
be regarded as purely kinematical, although the connection between
light and electrodynamics was advanced already in the mid 1800's.
The final outcome of these efforts has been that the ``speed of light''
is constant and also that this constancy is incompatible with Newtonian
mechanics. In this Section we will show that under the present approach
based on NAP the last assertion is incorrect. 

\subsubsection{The view of this work}

In the first place, we must raise the objection that light, being
a perception resulting from the electromagnetic interaction, is not
an \emph{object.} Hence, what is meant by ``the speed of light''
needs to be explained. What is truly measured in Fizeau's or Cornu's
experiments is a distance and a time-interval. The quotient of both
has the dimensions of a velocity so we may agree to consider it a
velocity. A view compatible with NAP, is to consider that all three
events related to the measurements (emission, reflection and detection)
are actually the result of a common cause, namely, the electromagnetic
interaction. 

Whatever we make in our interpretation of light, it is undeniable
that what we have called the speed of light is the quotient between
an objective distance and an intersubjective time interval. The result
is objective (or rather intersubjective) and it is subject to the
laws of transformations of objective quantities, this is, it is the
same for all observers. 

\subsubsection{The (traditional) mediator view}

The main step in the traditional conception of light is to interpret
the measured quotient in terms of the velocity of a body or a material
wave, something that has a place in space. Since interactions are
not matter, in principle this interpretation introduces a fundamental
belief: interactions are mediated by substantial entities that as
such have a place in space. As far as we know, neither this assumption
nor its rational/philosophical basis is ever stated. Interactions
do not have a place, interactions require two places and relative
distances, such as the relative distance between source/detector and
mirror. If we want to associate the events emission, reflection and
detection in terms of traditional causal relations it will have to
be in the order: turn on the light, reflect the light at the mirror,
detect the light. The cause (in the modern sense discussed at the
Introduction) of our detection of light cannot have ceased by the
time of the detection. Since without turning on the light there is
no light to be detected, and the same can be said when the mirror
is not in place, both the presence of the mirror and the turning on
of the light are within the determinants of the detected light. However,
the observed events are not the causes in themselves since the first
event might very well have ceased when the third event happens. Similarly,
the detected light immediately associated with the turn-on event could
only be a cause if it travelled (matter-like) in space. The interpretation
in terms of a material analogous allows for questions that do not
correspond to what it is truly measured.

Equating a subjective quantity such as $|v_{L}^{a}|$ (the velocity
of an interpreted material point $L$ with respect to a reference
$a$) with an objective quantity such as $C$ (namely a constant,
and hence objective and invariant) is, to say the least, confusing.
A possible construction could be to consider $\Delta t$ as subjective,
since after all, absolute time was an assumption of Newtonian mechanics.
If space is still to be conceived as absolute, then time might very
well depend on the velocity of the reference system with respect to
absolute space.

Let us work out an exercise along this line of thought. We will tentatively
assume as a new axiom a deeply rooted belief that we are not going
to offer for examination:

\begin{belief} Interactions are mediated by substantial entities that as such have a place in space and a velocity.\label{qm-entity}
\end{belief}

Then, if the substantial entity moves in absolute space with a characteristic
velocity, we would expect that it appears to us with different (and
measurable) relative velocities depending on the relative motion of
our system of references with respect to absolute space. However,
experiments (interpreted in this frame of mind) indicate otherwise:
we measure always the same velocity. Therefore, absolute time, absolute
space, these measurements and the present tacit belief have come into
collision. We can ask ourselves what can we save from the wreckage?
For the sake of the exercise, we insist in keeping the tacit belief
alive. It has been argued that such a state of things is incompatible
with Newtonian mechanics, but again, this is \textbf{not} the case.
Is it possible that velocities transform between different reference
systems satisfying the existence of a universal velocity \emph{and
}the structure of Newtonian mechanics? The following Theorems give
an affirmative solution. 
\begin{defn}
Let $U\in\left\{ z\in\mathbb{R}^{3}:|z|<C\right\} $ , with $|U|$
its Euclidean norm. Further, let $g(U)=G(|U|)$, where $G$ is a strictly
increasing continuous function such that $G(0)=1$ and $\lim_{a\to C^{-}}G(a)=\infty$.
We advance that $U$ will play the role of a velocity in the sequel.
Then, for $|U|,|V|<C$ we define \emph{velocity addition} $U\oplus V$
as 
\[
U\oplus V=W\Leftrightarrow g(W)W=g(U)U+g(V)V,
\]
where addition in the rhs denotes standard vector addition in 3-space.
\end{defn}
\begin{thm}
(Nonlinear presentation of Galileo's group) Velocity addition has
the following properties:
\end{thm}
\begin{enumerate}
\item $|U|,|V|<C\Rightarrow|W|<C$. 
\item $U\oplus0=U$, i.e., $0$ is the neutral element of velocity addition.
\item $U\oplus V=V\oplus U$.
\item For all $U$, the inverse of $U$ is $-U$, i.e., $U\oplus\left(-U\right)=0$.
\item Associativity: $\left(W\oplus V\right)\oplus U=W\oplus\left(V\oplus U\right)$.
\end{enumerate}
\begin{proof}
For most statements it is a matter of doing the algebra. For statement
a), the function $g(V)|V|$ is a bijection of the interval $[0,C)$
onto the nonnegative reals. $|U|,|V|<C\Rightarrow g(U)U+g(V)V<\infty$,
hence there is a unique $|W|$ in $[0,C)$ such that $g(W)W=g(U)U+g(V)V$.
Statements b), c) and d) follow from the definition. Regarding statement
e), let $A=W\oplus V$ and $B=V\oplus U$. Applying the definition
two times we have: $A\oplus U=g(A)A+g(U)U=g(W)W+g(V)V+g(U)U$. Similarly,
$W\oplus B=g(W)W+g(B)B=g(W)W+g(V)V+g(U)U$.
\end{proof}
In this exercise we consider relative position and velocity as primary
concepts. The addition of time intervals has to be deduced from the
concepts of relative position and velocity.
\begin{defn}
We call \emph{velocity} $V$ (with $|V|<C$ ) the vector satisfying
$\Delta x=V\Delta t$. For any reference system, we call $t$ \emph{subjective
time}, while $T=\frac{t}{g(V)}$ is called \emph{proper time}.
\end{defn}
The following Theorem indicates that while this way of thinking rests
on a subjective time, there exists an underlying intersubjective time
as long as space is objective. 
\begin{thm}
Let $O$ move with velocity $V_{1}$with respect to $S$ and velocity
$V_{2}$ relative to the observer $S^{\prime}$, while $S^{\prime}$
moves with velocity $V_{3}$ relative to $S$, such that $V_{1}=V_{2}\oplus V_{3}$
. Let 
\begin{eqnarray*}
\Delta x & = & V_{1}\Delta t_{1}\\
\Delta x^{\prime} & = & V_{2}\Delta t_{2}\\
\Delta_{SS^{\prime}} & = & V_{3}\Delta t_{3}=-\Delta_{S^{\prime}S}
\end{eqnarray*}
and $\Delta T_{i}=\frac{\Delta t_{i}}{g(V_{i})}$ for $i=1,2,3$.
Then the proper time is invariant, $\Delta T_{1}=\Delta T_{2}=\Delta T_{3}$,
if and only if distances in space are invariant.
\end{thm}
\begin{proof}
The distance between $S$ and $O$ can be determined both by $S$
and by $S^{\prime}.$ Invariance of distances is expressed as
\[
d_{SO}=\Delta x=\Delta x^{\prime}+\Delta_{SS^{\prime}}=d_{SO}^{\prime}
\]
or, equivalently,
\[
V_{1}=V_{2}\frac{\Delta t_{2}}{\Delta t_{1}}+V_{3}\frac{\Delta t_{3}}{\Delta t_{1}}=V_{2}\frac{g(V_{2})}{g(V_{1})}+V_{3}\frac{g(V_{3})}{g(V_{1})}.
\]
The last equality follows from the law of addition of velocities.
Invariance of distance holds for all velocities $U,V,W$ satisfying
$V=U\oplus W$, if and only if
\begin{eqnarray*}
\frac{\Delta t_{2}}{g(V_{2})} & = & \frac{\Delta t_{1}}{g(V_{1})}\\
\frac{\Delta t_{3}}{g(V_{3})} & = & \frac{\Delta t_{1}}{g(V_{1})}.
\end{eqnarray*}
It follows that $\Delta T_{1}=\Delta T_{2}=\Delta T_{3}$. The reverse
implication follows by assuming $\Delta T_{1}=\Delta T_{2}=\Delta T_{3}$
and applying the previous reasoning in reverse order to obtain $d_{SO}=d_{SO}^{\prime}$.
\end{proof}
$\Delta T$ can be interpreted as the time measured by an observer
moving with the object, since in such a case we have $g(0)=1$ and
$\Delta T=\Delta t$. 

As far as we know, the physics emerging from this picture has not
been studied. Certainly, we may imagine other alternatives to exit
the contradictory situation. For instance, we could drop entirely
Tacit belief \ref{qm-entity}, as in the view of this work, or try
other forms of reconciliation such as e.g., keeping Tacit belief \ref{qm-entity}
while dropping both absolute space and time, as well as relative distance.
This part of the story is well known.

We have shown that the argued need to abandon the Newtonian frame
upon the requirement of a constant speed of light is not mandatory.
On the contrary, this course of action is related to a bolder decision:
to associate interactions with the exchange of substantial entities
(i.e., to consider light an object) and subsequently opting for a
solution without all three of absolute space, absolute time and invariant
relative distance. 

\subsection{Poincaré's principle of relativity and related principles}

As we have already mentioned, the No Arbitrariness Principle relates
to Leibniz' principle of sufficient reason. Mach also used the principle
of sufficient reason in the form of a ``principle of symmetry''
\citep[p. 9]{mach19}. At the beginning of the XX-th century, a related
principle became notorious, the \emph{principle of relativity} presented
by \citet[p. 107]{poin13}. ``The motion of any system must obey
the same laws, whether it be referred to fixed axes or to moving axes
carrying along a rectilinear uniform motion. This is the principle
of relative motion that forces upon us for two reasons: first, the
commonest experiences confirm it, and second, the contrary hypothesis
is singularly repugnant to the mind''. Poincaré hints us about the
principle being a principle of knowledge (``not solely the result
of experiment''), but does not elaborate further. As we have shown,
Poincaré's principle belongs to the context of interactions, for in
the description of interactions it should not matter which isolated
body we take (arbitrarily) as reference for our calculations. This
issue was advanced before in this work. Poincaré and Mach inherit
the discussion between absolute and relative spaces already present
in Leibniz-Clarke and Newton. They attempt to resolve this issue in
different ways, most remarkably Mach reducing the cosmos to a set
of empirical relations without subjects (observers). The present approach
goes past this controversy by introducing the arbitrariness of the
observer (and with her/him the concept of space) and subsequently
eliminating the arbitrariness by focusing the understanding on objective
or intersubjective elements. In short, this work takes sides for an
objective and relative description of interactions, with the stress
resting on its objective nature.

\section{Summary and conclusions}

We have reconstructed Newtonian mechanics starting with a minimal
realism \citet{lomb12} that accepts as starting point the existence
of a real, or objective, world. The analysis of this physical world,
we called it Nature, proceeds through several constructive steps.
Some of them consist in dialectical openings to understanding. The
first opening is the opposition between \emph{objective} and \emph{subjective}:
objective opposes \emph{not-objective}, i.e., subjective, which is
what comes from the knowing subject. Understanding Nature is the result
of a dialectic interplay –or dialogue– between the knowing subject
and the object being known. It is worth mentioning that the subject
is not an individual in principle but rather humanity or society.
Knowledge in this primary form cannot be objective, the subject needs
to transcend (negate) her/his subjectivity if she/he is going to reach
any possible objectivity. Hence, the subject establishes principles
of reason such as the principle of sufficient reason, which we have
recast in mathematical terms under the name of No Arbitrariness Principle.
Actually, it is not the subject by her/himself who establishes the
principle but the subject in relation to Nature. There is no form
in which the dialectic objective-subjective can be separated giving
pre-eminence to one of the terms. 

NAP recognises the role of symmetry groups as a basic tool for understanding.
Symmetries contribute with more than just aesthetic appealing \citet{wign64};
they express the permanence of the laws of Nature beyond the arbitrariness
of the knowing subject. Moreover, NAP –or in a broader sense the principle
of sufficient reason– might have represented an adaptive advantage
for the early humanity, allowing to elaborate consequences of our
relation with Nature. 

NAP appears, at first sight, related to Kant's first formulation of
the \emph{categorical imperative} ``Act only according to that maxim
whereby you can, at the same time, will that it should become a universal
law'' \citet[(4:421)]{kant93}. The requisite of universal application
is precisely the action of removing all arbitrariness that may arise
in the application of the maxim to different rational agents. A moral
action must be free of arbitrariness. The  associated set of arbitrariness
usually considered is that of the agents subjected to the law. In
more common terms, we can say: no rational justice exists if the laws
do not pass the test of being equal for all citizens. We believe that
NAP is a fundamental principle of rationality that emerges in different
forms for different contexts.

Important contributions to mechanics come from the early childhood
where the notions of time and space are formed as a consequence of
the active adaptation of the child to her/his environment. In this
stage of our construction, concepts point toward perceptions rather
than to other concepts. With the opening produced by the concepts
of \emph{isolated} and \emph{interacting} comes the first mechanical
law, the only law acceptable to NAP for the condition of permanence
of isolated objects. Repeated use of NAP in different contexts will
finally unravel Newton's laws. But these laws require an increasing
number of assumptions regarding Nature. Indeed, Newton's mechanics
arises as the interplay of reason, idealisation and experience in
our relation to Nature.

We address the laws of Nature not as rules proper to the objects,
as in metaphysical realism \citep{lomb12}, but rather as laws for
the understanding of Nature generated by the interplay between subject
and object. This position becomes quite evident when Galilean transformations
of velocities are not alleged to be experimental laws but rather considered
as a requisite for the suppression of contributions arbitrarily introduced
by the observer.

The adopted approach allows us to attain a higher level of consciousness
on the strength of the different components that contribute to the
construction of mechanics. There is a precedence between these components
which can be better appreciated by considering the consequences of
removing them: (a) the consequence of renouncing to logic and to the
laws of understanding is not being able to understand the world, (b)
renouncing to the early elaborations of primary concepts such as time
and space leads to a dissociation between everyday life and physics,
the latter becoming entirely pragmatic and justified \emph{a-posteriori}
(because it is convenient), (c) in contrast, modifying our temporary
beliefs has no real cost other than the effort of reconstructing our
understanding on a more solid basis. Moreover, the present approach
allows for a critical view of further developments, since it opens
for new alternatives to the current views in different issues, alternatives
that we believe should be explored to decide their worth.

An important reflection corresponds to the use of tacit beliefs (also
known, in mathematics, as \emph{hidden assumptions} or \emph{hidden
lemmas}). Inasmuch as they are hidden and as such can escape the scrutiny
of reason, they must be considered dangerous \citet{liss17}. They
elude the principle of necessary reason and usually protect something
that is not explicitly mentioned. In the case of Tacit belief \ref{qm-entity},
it hides the fact that we support the argument in a material model,
be it wave or particle, on the basis that we cannot imagine it otherwise.
Our imagination works with memories of the sensible world and in doing
so it limits its value to the material experience. ``Because we are
not able to imagine otherwise'' appears to us as an extremely weak
argument: the confession of a limitation does not constitute a reason.
We can explore the possibilities of the hidden assumption but we should
be ready to abandon it if it comes into conflict with higher principles.

Scientific disciplines are not autonomous, they are ruled by reason,
which is the only autonomous entity. The application of the principles
of reason to understanding is known as \emph{critic}. Critic in science
needs not only be declared but it should be truly exercised. This
work is indeed an exercise of critic. It seems proper to quote Kant
\citep[The discipline of reason]{kant87} on this respect ``... where
reason is not held in a plain track by the influence of empirical
or of pure intuition, that is, when it is employed in the transcendental
sphere of pure conceptions, it stands in great need of discipline,
to restrain its propensity to overstep the limits of possible experience
and to keep it from wandering into error. In fact, the utility of
the philosophy of pure reason is entirely of this negative character.''
The terms pure intuition and pure conceptions in Kant may be understood
from: ``Pure intuition consequently contains merely the form under
which something is intuited, and pure conception only the form of
the thought of an object. Only pure intuitions and pure conceptions
are possible a priori; the empirical only \emph{a posteriori}.''
\citep[p. 75]{kant87}

In conclusion, the construction of knowledge is not only a philosophical
perspective but it is a constructive and critic method as well. Indeed,
the history of knowledge has to recognise the contributions from metaphysical
(unverifiable/irrefutable) ideas such as absolute space and later
the reification of interactions. Changing the underlying metaphysics,
great revolutions operate in physics, which is not automatically equal
to great progresses. In these lines, recognising the need for introducing
some metaphysics, the idea of a pluralist realism \citet{chan12-realism,lomb12}
should be considered not only as a description but rather as a programme.
We might ask if the efforts of critical revision such as those of
Mach or the present work could have been possible without the preceding
developments and their metaphysical ingredients. Critic, as the negative
tool of philosophy, cannot operate in the void but can operate as
self-criticism in the process of construction of knowledge. 

\bibliographystyle{spbasic}
\bibliography{nuevasreferencias,referencias}

H. G. Solari is professor of physics at the School of Exact and Natural
Sciences, University of Buenos Aires and research fellow of the National
Research Council of Argentina (CONICET). He has contributed research
to physics, mathematics, biology (ecology and epidemiology) and more
recently has been engaged in Applied Philosophy of Sciences (defined
as a reflexive and critical practice of science grounded in philosophy).
Most of his work is of interdisciplinary type.

M. A. Natiello is professor of applied mathematics at the Centre for
Mathematical Sciences, Lund University. He has contributed research
in chemistry, physics, mathematics, biology. His current research
interests are Applied Philosophy of Sciences and Population Dynamics
with a focus on social, ecological and epidemiological aspects. 

\end{document}